\documentclass[11pt]{article}
\usepackage{latexsym,amsmath,amscd,amssymb,graphics,mathrsfs}
\usepackage{enumerate}
\usepackage{color}
\usepackage{url}
\usepackage[all]{xy}

\newcommand{\rem}[1]{}
\makeatletter

\newcommand \al{\alpha}
\newcommand\be{\beta}
\newcommand\ga{\gamma}

\newcommand\et{\eta}
\renewcommand\th{\theta}

\newcommand\la{\lambda}
\newcommand\rh{\rho}

\newcommand\si{\sigma}

\newcommand\ph{\varphi}

\newcommand\ps{\psi}
\newcommand\om{\omega}
\newcommand\Ga{\Gamma}

\newcommand\Th{\Theta}

\newcommand\Om{\Omega}

\newcommand\resp{resp.\ }
\newcommand\ie{i.e.\ }

\newcommand\oo{{\infty}}

\renewcommand\o{\circ}

\newcommand\x{\times}
\newcommand\on{\operatorname}
\newcommand\Ad{\on{Ad}}
\newcommand\ad{\on{ad}}

\newcommand\Emb{\on{Emb}}
\newcommand\symp{\on{symp}}
\newcommand\Rot{\on{Rot}}
\newcommand\ex{\on{ex}}
\newcommand\Den{\on{Den}}
\newcommand\hor{\on{hor}}
\newcommand\vol{\on{vol}}
\newcommand\iso{\on{iso}}

\newcommand\N{\mathcal{N}}
\newcommand\M{\mathcal{M}}

\newcommand\Diff{\on{Diff}}
\newcommand\per{{\perp_g}}

\newcommand\Vol{\on{Vol}}
\newcommand\ii{{\bf i}}
\newcommand\pr{\on{pr}}

\newcommand\ham{\on{ham}}
\newcommand\quant{\on{quant}}
\newcommand\Lag{\on{Lag}}
\newcommand\g{\mathfrak g}

\newcommand\h{\mathfrak h}
\newcommand\Gr{\on{Gr}}

\newcommand\dd{{\bf d}}

\newcommand\ZZ{\mathbb Z}

\newcommand\RR{\mathbb{R}}

\newcommand\J{\mathbf{J}}
\newcommand\X{\mathfrak X}
\newcommand\C{\mathcal C}

\renewcommand\O{\mathcal O}

\newcommand\U{\mathcal U}

\newenvironment{proof}[1][Proof]{\noindent\textbf{#1.} }{\ \rule{0.5em}{0.5em}}

\begin{document}

\newtheorem{theorem}{Theorem}[section]
\newtheorem{definition}[theorem]{Definition}
\newtheorem{lemma}[theorem]{Lemma}
\newtheorem{remark}[theorem]{Remark}
\newtheorem{proposition}[theorem]{Proposition}
\newtheorem{corollary}[theorem]{Corollary}
\newtheorem{example}[theorem]{Example}


\title{Isotropic submanifolds and coadjoint orbits\\ of the Hamiltonian group}
\author{Fran\c{c}ois Gay-Balmaz$^{1}$ and Cornelia Vizman$^{2}$ }

\addtocounter{footnote}{1}
\footnotetext{CNRS/LMD, \'Ecole Normale Sup\'erieure de Paris, France.
\texttt{francois.gay-balmaz@lmd.ens.fr }
\addtocounter{footnote}{1}}

\footnotetext{Department of Mathematics,
West University of Timi\c soara, 
300223-Timi\c soara, Romania.
\texttt{vizman@math.uvt.ro}
\addtocounter{footnote}{1} }

\date{ }
\maketitle
\makeatother




\begin{abstract}
We describe a class of coadjoint orbits of the group of Hamiltonian diffeomorphisms of a symplectic manifold $(S,\om)$ by implementing symplectic reduction for the dual pair associated to the Hamiltonian description of ideal fluids. The description is given in terms of nonlinear Grassmannians (manifolds of submanifolds) with additional geometric structures. Reduction at zero momentum yields the identification of coadjoint orbits with Grassmannians of isotropic volume submanifolds, {slightly} generalizing the results in \cite{W90} and \cite{L}. At the other extreme, the case of a nondegenerate momentum recovers the identification of connected components of the nonlinear symplectic Grassmannian with coadjoint orbits, thereby \textcolor{black}{recovering} the result of \cite{Haller-Vizman}. 
\textcolor{black}{We also comment on the intermediate cases which correspond to new classes of coadjoint orbits}.
The description of these coadjoint orbits as well as their orbit symplectic form is obtained in a systematic way by exploiting the general properties of dual pairs of momentum maps.
We also show that whenever the symplectic manifold $(S,\om)$ is prequantizable, 
the coadjoint orbits that consist of isotropic submanifolds with 
total volume $a\in\ZZ$ are prequantizable. The prequantum bundle is constructed explicitly and, in the Lagrangian case, recovers the Berry bundle constructed in \cite{W90}.
\end{abstract}



\section{Introduction and preliminaries}\label{s1}

This paper concerns the description and prequantization of a class of infinite dimensional coadjoint orbits of the group $ \operatorname{Diff}_{\rm ham}(S)$ of Hamiltonian diffeomorphisms of a symplectic manifold $(S, \omega )$ and of its central extension, the identity component of the group $ \Diff_{\quant}(P)$ of quantomorphisms of  the prequantum bundle $P \rightarrow S$, when $ \omega $ is prequantizable. \textcolor{black}{We obtain our results by a systematic use of the process of symplectic reduction applied to the dual pair of momentum maps associated to the Hamiltonian description of ideal fluids}.

Several descriptions of classes of coadjoint orbits have been already given.
In \cite{W90} a foliation of the space of Lagrangian submanifolds of $S$,
whose leaves  consist of Lagrangian submanifolds that can
be joined by flowing along Hamiltonian vector fields, is considered,
together with a corresponding isodrastic foliation of the space of weighted Lagrangian submanifolds (here a weight is a smooth density of total measure 1).  
It is argued  heuristically that the leaves that consist of positively weighted Lagrangian submanifolds can be identified with coadjoint orbits of the group of Hamiltonian diffeomorphisms. 
All these facts are showed rigorously in \cite{L}, where these coadjoint orbits
are obtained by {symplectic reduction} on the manifold of embeddings into $S$.
More general coadjoint orbits, that consist of positively weighted isotropic submanifolds, are also obtained.
In \cite{Haller-Vizman} another class of coadjoint orbits of the group of Hamiltonian diffeomorphisms was identified with connected components of the nonlinear  Grassmannian of symplectic submanifolds of $S$. 

\medskip

In the present paper, we describe in terms of nonlinear Grassmannians a class of coadjoint orbits of the group of Hamiltonian diffeomorphisms of a symplectic manifold by implementing symplectic reduction for the dual pair of momentum maps associated to the Hamiltonian description of ideal fluids \cite{Marsden-Weinstein}. This class contains as particular cases the above two descriptions of coadjoint orbits.

To obtain this result, we use the reformulation of the dual pair of ideal fluids given in \cite{GBVi2011} that allows a rigorous proof of the dual pair properties. This reformulation consists in restricting the action of symplectic, \resp volume preserving, diffeomorphisms to the subgroups of Hamiltonian, \resp exact volume preserving, diffeomorphisms, and to consider the prequantization and Ismagilov central extensions of these subgroups, respectively. As a consequence of our approach, we need to impose an extra condition, namely, the vanishing of the first cohomology of the submanifolds or the exactness of the symplectic form. 

The dual pair property is crucially used to identify the symplectic reduced space relative to one group action with coadjoint orbits of the other group. To this end, we need to formulate two general results on dual pairs of momentum maps, that extend the results of \cite{Balleier-Wurzbacher} and that apply to the ideal fluid dual pair.

In the case of the class of coadjoint orbits 
\textcolor{black}{of weighted isotropic submanifolds} obtained in \cite{W90} and \cite{L}, our approach yields much concrete expressions for the tangent spaces to the coadjoint orbits and hence, explicit formulas for the orbit symplectic form.
Assuming that the symplectic manifold $(S,\om)$ is prequantizable, with prequantum bundle $P\to S$,
we also show that these coadjoint orbits are prequantizable whenever
the total volume $a$ of the isotropic submanifolds is an integer, and construct explicitly the prequantum bundle. It is the space of horizontal submanifolds of $P$ that cover isotropic submanifolds of $S$ in the coadjoint orbit, factorized by the action
of $a$th roots of unity induced by the principal circle action on $P$.
In the Lagrangian case we get Planckian submanifolds \cite{Souriau} (Legendre submanifolds of the contact manifold $P$), thus \textcolor{black}{generalising} the result from \cite{W90}
that the nonlinear Grassmannian  of weighted Lagrangian submanifolds 
 is prequantizable \textcolor{black}{when $a=1$}. In this case, the prequantum bundle recovers the Berry bundle of \cite{W90}.


\color{black} 
\paragraph{Contribution of the paper.} To guide the reader and provide a quick overview of our results, we briefly list here the main contributions of the paper.
\begin{itemize}
\item We present a new and systematic derivation for the identification of coadjoint orbits of the group of Hamiltonian diffeomorphisms with particular nonlinear Grassmannians, see Theorems \ref{unul} and \ref{doiu}. Our approach unifies earlier identifications, given in \cite{W90} and \cite{L} for  spaces of isotropic volume submanifolds \textcolor{black}{with volume one} and in \cite{Haller-Vizman} for spaces of symplectic submanifolds. \textcolor{black}{We also slightly extend the class of coadjoint orbits identified in these earlier works by considering Grassmannians of isotropic volume submanifolds of arbitrary fixed volume and Grassmannians of (weighted) presymplectic submanifolds}. Our approach allows a concrete description of the tangent spaces of the Grassmannians and of the orbit symplectic forms. 

\item This derivation is based on the recognition that certain dual pairs of momentum maps yield, under some conditions, an isomorphism between symplectically reduced spaces and coadjoint orbits, as explained in Section \ref{s4}. The dual pair of momentum maps that is fundamental for this work, is the dual pair associated to the Hamilton description of ideal fluids, as explained in Section \ref{s5}.

\item We prove that, under natural conditions, these coadjoint orbits are prequantizable and we concretely describe the prequantum bundle. 
This result extends to the isotropic case the results obtained in \cite{W90} for the Lagrangian case. This is the content of Section \ref{s7}.

\item We describe the coadjoint orbits and their prequantification for the case of an exact symplectic manifold, see Theorem \ref{tre} and Proposition \ref{prequant_exact}. In this case we can treat the case of submanifolds with nontrivial cohomology.
\end{itemize}
\color{black}

The goal of this section is to carry out symplectic reduction for the right momentum map  of the ideal fluid dual pairs \eqref{MW_GBVi_dual_pair}, \resp \eqref{MW_GBVi_dual_pair_compact}, in order to describe coadjoint orbits of the group of Hamiltonian diffeomorphisms in terms of nonlinear Grassmannians. The case of zero momentum allows us to obtain Grassmannians of isotropic volume submanifolds as coadjoint orbits, slightly generalizing the results in \cite{W90} and \cite{L}. At the other extreme, the case of a nondegenerate momentum recovers the identification of connected components of the nonlinear symplectic Grassmannian with coadjoint orbits, thereby recovering the result of \cite{Haller-Vizman}.
Some comments about the intermediate cases will be given. Our approach naturally yields concrete expressions for the tangent spaces to the orbits and hence, explicit formulas for the orbit symplectic form.



\paragraph{Plan of the paper.} In the remainder of this Introduction, we recall some facts concerning nonlinear Grassmannians and we review 
two central extensions of groups of diffeomorphisms.
\textcolor{black}{At the end of the introduction we provide a glossary that contains a list of notations for most of the mathematical objects used throughout this article.}
In Section \ref{s4} we consider two situations, relevant in the infinite dimensional setting, in which symplectic reduction in a dual pair of momentum maps
provides coadjoint orbits.
In Section \ref{s5} we review the ideal fluid dual pair on $\Emb(M,S)$ 
when $H^1(M)=0$. 
Symplectic reduction with respect to the right action in the ideal fluid dual pair is done in Section \ref{s6}
and is used to identify each connected component of the nonlinear Grassmannian of  isotropic volume submanifolds with a coadjoint orbit.
In Section \ref{s7}, we show that their connected components are prequantizable
coadjoint orbits, when $S$ is a prequantizable symplectic manifold.
Finally, in Section \ref{s8},
starting from  the ideal fluid dual pair on $\Emb(M,S)$ when $S$ is exact symplectic,  we redo everything for exact isotropic volume submanifolds.
The appendix  discusses the Fr\'echet manifold structure on several 
nonlinear Grassmannians.


\paragraph{Nonlinear Grassmannians.} Let $S$ be a manifold, let $M$ be a compact $k$-dimensional manifold, and consider the Fr\'echet manifold $\Emb(M,S)$ of all embeddings  of $M$ into $S$. The tangent space at $f\in\Emb(M,S)$ is the space of vector fields 
on $S$ along $f$. The group $\Diff(M)$ of diffeomorphisms of $M$ acts on $\Emb(M,S)$ by composition on the right.
The associated quotient map $\pi:f\in\Emb(M,S)\mapsto f(M)\in\Gr^M(S)$ 
defines a principal bundle with structure group $\Diff(M)$
over the \textit{nonlinear Grassmannian $\Gr^M(S)$ of (embedded) submanifolds of $S$ of type $M$},
which  is known to be a Fr\'echet manifold \cite{KrMi97} \cite{Molitor} (see also the appendix).
The tangent space to $ \operatorname{Gr}^M(S)$ at $N=f(M)$ 
is given by the space of smooth sections of the normal bundle $TN^\perp:=(TS|_N)/TN$. The tangent map to the projection $ \pi $ reads
\begin{equation}\label{bigp}  
T_f \pi : v \circ f \in T_f \operatorname{Emb}(M,S) \mapsto v|^{\perp}_{f(M)} \in T_{f(M)} \operatorname{Gr}^M(S), \quad v \in \mathfrak{X}  (S).
\end{equation}

The \textit{nonlinear Grassmannian of volume submanifolds of type $(M,\mu)$} is defined by
\[
{\Gr^{M,\mu}(S):=\left\{(N,\nu) : N\in  \Gr^M(S), \nu\in\Vol(N),\int_N\nu=\int_M\mu\right\}.}
\]
The map 
\begin{equation}\label{pim}
\pi^\mu:f\in\Emb(M,S)\;\longmapsto\; (f(M),f_*\mu)\in\Gr^{M,\mu}(S)
\end{equation}
is a principal $\Diff_{\vol}(M)$-bundle and the forgetting map $(N, \nu ) \in \Gr^{M,\mu}(S)\to N \in \Gr^{M}(S)$ is a fiber bundle with fiber $\Vol_1(M)$, the space of volume forms of total volume 1. Using a Riemannian metric on $S$, the tangent space to $\Gr^{M,\mu}(S)$ at $(N, \nu )$ is identified with
$T_N\Gr^M(S)\x \mathbf{d} \Om^{k-1}(N)$.
We refer to \cite{GBVi2013} for a detailed study of the Fr\'echet manifold structures, together with the treatment of the more general case when $ \partial M\neq \varnothing$.

\textcolor{black}{We summarise in Table \ref{table} the list of all the infinite dimensional manifolds occurring in the paper.}


\paragraph{The prequantization central extension.}
Let $(S,\om)$ be a prequantizable symplectic manifold, 
\ie there exists a principal circle bundle (the prequantum bundle) $p:P\to S$ with 
principal connection $\al\in\Om^1(P)$ whose curvature is the symplectic form  $\om$, so $\mathbf{d} \al=p^*\om$.
Let us denote by $\Diff_{\ham}(S)$ the group of \textit{Hamiltonian diffeomorphisms} of $(S, \omega )$. Its Lie algebra is the space $\X_{\ham}(S)=\{X_f \in \mathfrak{X}  (S) : \mathbf{i} _{X_f}\om=\dd f, \; f\in C^\oo(S) \}$ of all Hamiltonian vector fields. The \textit{quantomorphism group} is 
the group $\Diff_{\operatorname{quant}}(P)$
of connection preserving automorphisms of $P$. Its Lie algebra is $C^\oo(S)$, endowed with the Poisson bracket  $\{f,g\}=\om(X_g,X_f)$, under the identification of $h\in C^\oo(S)$ with $\xi_h=X_h^{\hor}-(h\o p)E$, where $E$ denotes the infinitesimal generator of the circle action
and hor the horizontal lift.

For connected $S$, the Lie algebra extension 
\begin{equation}\label{q1}
0\to\RR\to C^\oo(S)\to\X_{\ham}(S)\to 0
\end{equation}
integrates to the \textit{prequantization central extension} of $\Diff_{\ham}(S)$ given by \cite{Kostant}\cite{Souriau}
\begin{equation}\label{unu}
1\to S^1\to\Diff_{\quant}(P)_0\to\Diff_{\ham}(S)\to 1,
\end{equation}
where $\Diff_{\quant}(P)_0$ denotes the component of the identity. A version of this exact sequence of groups
for infinite dimensional $S$ can be found in \cite{NV}.
{The universal central extension of the Lie algebra of Hamiltonian vector fields can be found in \cite{JV1}.
Integrability issues were posed in \cite{JV2}.}
 
\paragraph{Ismagilov's central extension.} Let $M$ be a compact $k$-dimensional manifold with volume form $\mu$ and let $\Diff_{\vol}(M)$ be the group of \textit{volume preserving diffeomorphisms} with Lie algebra $\X_{\vol}(M)$ of divergence free vector fields.
We denote by $\Diff_{\ex}(M)$ the subgroup of \textit{exact volume preserving diffeomorphisms}
with Lie algebra $\X_{\ex}(M)$, the Lie algebra of vector fields $X_\al$ admitting a potential form $\al\in\Om^{k-2}(M)$, \ie
$\mathbf{i} _{X_\al}\mu=\dd\al$ \cite{Banyaga}\cite{KrMi97}. 
If $\dim M=2$, then the volume form 
is a symplectic form and we are in the previous paragraph setting.
\color{black}In \cite{Hitchin} it is shown that, when $\dim M=3$, the group of diffeomorphisms
that preserve the equivalence class of a gerbe with curvature $\mu$
also integrates the Lie algebra $\X_{\ex}(M)$.

Assume that $\dim M\ge 3$. \color{black} If $ \mu $ is integral ($\int_M \mu \in \mathbb{Z}  $), the Lichnerowicz Lie algebra extension \cite{Roger}
\begin{equation}\label{q2}
0\to H^{k-2}(M)\to\Om^{k-2}(M)/\dd\Om^{k-3}(M)\to\X_{\ex}(M)\to 0,
\end{equation}
with Lie algebra bracket 
\begin{equation}\label{iimu}
\{[\al],[\be]\}=[ \mathbf{i} _{X_\al} \mathbf{i} _{X_\be}\mu]\text{ on }\Om^{k-2}(M)/\dd\Om^{k-3}(M),
\end{equation}
integrates to \textit{Ismagilov's central extension} \cite{Ismagilov}
\begin{equation}\label{doi}
1\to H^{k-2}(M)/L^*\to\widehat\Diff_{\ex}(M)\to\Diff_{\ex}(M)\to 1,
\end{equation}
where $L^*$ is the dual lattice to the lattice $L\subset H_{k-2}(M,\RR)$ generated by 
a fixed basis of $H_{k-2}(M,\RR)$ consisting of co-dimension two submanifolds of $M$.

\begin{table}[h!]
\centering
 \begin{tabular}{|c | c |} 
 \hline\hline
 Notation & Name \\ [0.25ex] 
 \hline\hline
 $\Diff_{\vol}(M)$ & The group of volume preserving diffeomorphisms of the volume manifold $M$  \\ [0.25ex] 
\hline
 $\Diff_{\ex}(M)$ & The group of exact volume preserving diffeomorphisms of the volume manifold $M$  \\ [0.25ex] 
\hline
 $\Diff_{\symp}(S)$ & The group of symplectic diffeomorphisms 
of the symplectic manifold $S$  \\ [0.25ex] 
\hline
$\Diff_{\ham}(S)$ & The group of Hamiltonian diffeomorphisms 
of the symplectic manifold $S$  \\ [0.25ex] 
\hline\hline
 $\Emb(M,S)$ & The space of embeddings of $M$ into $S$  \\  [0.25ex] 
\hline
 $\Gr^M(S)$ & The Grassmannian of submanifolds of $S$ of type $M$ \\ [0.25ex] 
\hline
 $\Gr_{\symp}^M(S)$ & The Grassmannian of symplectic submanifolds of $S$ of type $M$ \\ [0.25ex] 
\hline
 $\Gr^{(M,\mu)}(S)$ & The Grassmannian of volume submanifolds of $S$ of type $(M,\mu)$ \\ [0.25ex] 
\hline\hline
 $\Emb_{\iso}(M,S)$ & The space of isotropic embeddings of $M$ into $S$  \\  [0.25ex] 
\hline
 $\Gr_{\iso}^M(S)$ & The Grassmannian of isotropic submanifolds of $S$ of type $M$ \\ [0.25ex] 
\hline
 $\Gr_{\iso}^{(M,\mu)}(S)$ & The Grassmannian of isotropic volume submanifolds of $S$ of type $(M,\mu)$ \\ [0.25ex] 
\hline
 $\Emb_{\Lag}(M,S)$ & The space of Lagrangian embeddings of $M$ into $S$  \\  [0.25ex] 
\hline
 $\Gr_{\Lag}^M(S)$ & The Grassmannian of Lagrangian submanifolds of $S$ of type $M$ \\ [0.25ex] 
\hline
 $\Gr_{\Lag}^{(M,\mu)}(S)$ & The Grassmannian of Lagrangian volume submanifolds of $S$ of type $(M,\mu)$ \\ [0.25ex] 
\hline\hline
 $\Diff_{\quant}(P)$ & The quantomorphism group of the prequantum bundle $P$  \\ [0.25ex] 
\hline
 $\Emb_{\hor}(M,P)$ & The space of horizontal embeddings of $M$ into $P$  \\  [0.25ex] 
\hline
 $\Gr_{\hor}^M(P)$ & The Grassmannian of horizontal submanifolds of $P$ of type $M$ \\ [0.25ex] 
\hline
 $\Gr_{\hor}^{(M,\mu)}(P)$ & The Grassmannian of horizontal  volume submanifolds of $P$ of type $(M,\mu)$ \\
 [0.25ex] 
 \hline\hline
 \end{tabular}
\caption{Glossary}\label{table} 
\end{table}

\section{Symplectic reduction in dual pairs}\label{s4}

In this section, we first review the definition of a dual pair by mainly focusing on the particular case of dual pair of momentum maps. 
We first recall from \cite{Balleier-Wurzbacher}  that the reduced symplectic manifolds relative to one of the group action are symplectically diffeomorphic to coadjoint orbits of the other group (Proposition \ref{till}).
Then, we extend this result to specific situations  (Propositions \ref{id1} and \ref{id2}) needed for the treatment of the dual pair associated to the Hamiltonian formulation of ideal fluids \cite{Marsden-Weinstein}, \cite{GBVi2011}.

\paragraph{Dual pairs.}
Let $(S,\omega)$ be a symplectic manifold and $P,Q$ be two Poisson manifolds. A pair of Poisson mappings
\begin{equation*}
P\stackrel{\J_L}{\longleftarrow}(S,\om)\stackrel{\J_R}{\longrightarrow} Q
\end{equation*}
is called a \textit{dual pair} \cite{Weinstein}
if $\ker T\J_L$ and $\ker T\J_R$ are symplectic orthogonal complements of one another:
$(\ker T\J_L)^\om=\ker T\J_R$.
In infinite dimensions, due to the weakness of the symplectic form, one has to impose also the identity $(\ker T\J_R)^\om=\ker T\J_L$ \cite{GBVi2011}.

Suppose now that the two Poisson mappings are momentum maps $\J_R$ and $\J_L$
arising from the commuting Hamiltonian actions of two Lie groups $G$ and $H$ on $S$.
To fix ideas, we will always assume that $H$ acts on the left and $G$ acts on the right on $S$, hence the notations $\J_L$ and $\J_R$. We assume that both momentum maps are equivariant, so that they are Poisson maps with
respect to the Lie-Poisson structure on the dual Lie algebras $\h^*$ and $\g^*$.
The actions are said to be {\it mutually completely orthogonal} \cite{Libermann-Marle} if 
the $G$- and $H$-orbits are symplectic orthogonal to each other:
\begin{equation}\label{mutcomort}
\mathfrak{g}_S=\mathfrak{h}_S^{\om}\quad\text{and}\quad \mathfrak{h}_S=\mathfrak{g}_S^{\om},
\end{equation}
where $\mathfrak{g}_S(s):=\{\xi_S(s)\mid \xi\in\mathfrak{g}\}$
with $\xi_S$ denoting the infinitesimal generator. 
Because $\ker T\J_R=\g_S^\om$,
the identities \eqref{mutcomort} mean that the infinitesimal actions of $\g$ \resp $\h$ 
on level sets of momentum maps $\J_L$ \resp $\J_R$ are transitive.
Hence, if $G$, \resp, $H$ act transitively on (connected components of) level sets of the momentum maps $\mathbf{J}_L$, \resp, $\mathbf{J}_R$, 
then the identities in \eqref{mutcomort} are automatically satisfied.

\begin{remark}
{\rm 
In finite dimensions the identities \eqref{mutcomort} are equivalent to the fact that
\begin{equation}\label{dp}
\h^*\stackrel{\J_L}{\longleftarrow}(S,\om)\stackrel{\J_R}{\longrightarrow}\g^*
\end{equation}
is a dual pair.
This is not the case in infinite dimensions: the free boundary fluid dual pair is a counterexample with $\h_S\subsetneq\g_S^\om$,
as showed in \cite{GBVi2012}.
}
\end{remark}


\paragraph{Symplectic reduction in dual pairs.} 
We now describe the symplectic reduced space for one of the momentum maps (say $ \mathbf{J} _R $) of a dual pair 
associated to mutually completely orthogonal actions.  Throughout this section we assume that, given $\si\in\g^*$ with isotropy group $G_\si$ and isotropy Lie algebra $\g_\si$, the level set $ \mathbf{J} _R ^{-1} (\si)$ is a submanifold of $S$
and that $S_\si:=  \mathbf{J} _R ^{-1} (\si)/{G_\si}$ can be endowed with the quotient manifold structure. In finite dimensions, standard hypotheses can be imposed to guarantee these properties. Since we are working in infinite dimensions, we will need to verify these properties in each of the treated examples.

We remark that, by the symplectic orthogonality  conditions \eqref{mutcomort},
for all $x\in\J_R^{-1}(\si)$,
$$
T_x(\J_R^{-1}(\si))=\ker T_x\J_R=\g_S(x)^\om=\h_S(x).
$$
Since the momentum map $\J_R$ is $G$-equivariant, 
$$\om_x(\xi_S(x),\et_S(x))=-\langle\J_R(x),[\xi,\et]\rangle=-\langle\si,[\xi,\et]\rangle=0,\quad\forall\,\xi\in\g_\si,\;\forall\,\et\in \mathfrak{g},$$ 
so $(\g_\si)_S(x)\subset\g_S(x)^\om=\h_S(x)$. 
In particular, for all $x\in\J_R^{-1}(0)$, one obtains that $\g_S(x)$ is an isotropic subspace of $T_xS$.
Thus the tangent space to the reduced symplectic manifold $S_\si$
at the $G_\si$-orbit $[x]$ of $x\in\J_R^{-1}(\si)$ is given by the quotient vector space $T_{[x]}S_\si=\h_S(x)/(\g_\si)_S(x)$.
The reduced symplectic form on $S_ \sigma $ is
\begin{equation}\label{omega_0}
(\om_\si)_{[x]}([\xi_S(x)],[\et_S(x)])=\om_x (\xi_S(x),\et_S(x))
=\langle\J_L(x),[\xi,\et]\rangle,
\end{equation}
for all $\xi,\et\in\h$. It is well defined since $\om$ vanishes on pairs of mixed generators for the two actions by \eqref{mutcomort}.

The next result can be extracted from Theorem 2.8 in \cite{Balleier-Wurzbacher}.
Although it is shown there for Howe dual pairs of momentum maps, 
the proof works for the following setting, in infinite dimensions too.

\begin{proposition}\label{till}
Let $G$ and $H$ be Lie groups with commuting mutually orthogonal symplectic actions on
the symplectic manifold $(S,\om)$. {We assume that these actions admit equivariant momentum maps}.
If the level sets of the momentum maps of each action are the orbits of the other action,
then the symplectic reduced spaces for one action are symplectically diffeomorphic to the coadjoint orbits for the other one.
\end{proposition}

{Note that although we are interested in the coadjoint orbits of one group, say $H$, we still need the transitivity hypotheses for both group actions.}

In order to identify reduced symplectic manifolds with coadjoint orbits in the context of the dual pair associated to ideal fluids (see Section \ref{s5}), we need to consider a slightly different setting, namely the case where one of the groups does not act transitively on the level sets of the momentum map, but acts transitively on \textit{the connected components} of the level sets. We will use the following observation:

\begin{lemma}\label{easy} Consider a transitive $G$-action on a symplectic manifold $(S, \omega )$. Suppose that this action admits an equivariant and injective momentum map $ \mathbf{J} :S \rightarrow \mathfrak{g}  ^\ast $. Then 
the image of $ \mathbf{J} $ is a coadjoint orbit of $G$. Moreover, the pull-back of the orbit symplectic form by $\J$ is $ \omega $.
\end{lemma}

\color{black}
\begin{proposition}\label{id1}
Let $(S,\om)$ be a symplectic manifold with commuting mutually orthogonal symplectic actions of $G$ and $H$, actions that admit equivariant momentum maps.
We assume that $G$ acts transitively on the level sets of $\J_L$, while
$H$ is connected and acts transitively on the connected components of level sets of $\J_R$.

Then each connected component of the reduced symplectic manifold
$S_\si=\J_R^{-1}(\si)/G_\si$  at $\si\in\g^*$, with $G_\si$ the stabilizer of $\si$,
is symplectically diffeomorphic to a coadjoint orbit of $H$.
\end{proposition} 
\color{black}

\begin{proof}
Since the two actions on $S$ commute, the reduced manifold $S_\si$ still admits a Hamiltonian $H$-action
with injective equivariant momentum map given by
\begin{equation}\label{jele}
\bar\J_L:S_\si=\J_R^{-1}(\si)/G_\si\to\h^*, \quad \bar\J_L([x])=\J_L(x).
\end{equation}
It is $H$-equivariant by the $H$-equivariance of $\J_L$
and injective because $G$ acts transitively on the level sets of $\J_L$.
Indeed, $\bar\J_L([x_1])=\bar\J_L([x_2])$ with $x_1,x_2\in\J_R^{-1}(\si)$
implies the existence of $g\in G$ with $g\cdot x_1=x_2$.
In particular, by the $G$-equivariance of $\J_R$, we have $\Ad^*_g(\si)=\si$,
so $g\in G_\si$ and $[x_1]=[x_2]$.

Let $\C$ be the connected component of the reduced manifold $S_\si$ containing the element $[x_0]$. 
If we show that $H$ acts transitively on $\C$, then  Lemma \ref{easy} implies
that $\C$, with its reduced symplectic form,
can be identified via $\bar\J_L$
with a coadjoint orbit of $H$, endowed with the orbit symplectic form.

We know that $H$ acts transitively on the connected components of
the level set $\J_R^{-1}(\si)$,
so, since $H$ is connected, the decomposition of $\J_R^{-1}(\si)$ into $H$-orbits
coincides with its decomposition into connected components.
It follows that the bi-orbit $B_0=H\cdot x_0\cdot G$, being $H$-saturated, 
is an open subset of $\J_R^{-1}(\si)$.
Let $\pi:\J_R^{-1}(\si)\to S_\si$ denote the canonical projection.
The complement $B_1=\pi^{-1}(\C)-B_0$ is $H$-saturated, hence open in $\J_R^{-1}(\si)$.
Indeed, for $x\in B_1$ and $h\in H$, we see that $h\cdot x\notin B_0$ (since $s\notin B_0$)
and $\pi(h\cdot x)\in\C$ (since $\pi(H\cdot  x)\subset\C$, for 
$\pi(x)\in\C$).

We get a disjoint decomposition of the connected set $\C$ into two open subsets 
$\pi(B_0)\ne\varnothing$ and $\pi(B_1)$.
Indeed, on one hand, $B_0$ and $B_1$ are both open and $G$-saturated, 
hence $\pi(B_0)$ and $\pi(B_1)$ are open,
on the other hand they have empty intersection 
because $B_0$ and $B_1$ are $G$-saturated with empty intersection.
It follows that the connected component $\C$ is the $H$-orbit $\pi(B_0)=H\cdot [x_0]$.
\end{proof}

\paragraph{Non-equivariant momentum map.}
We will also need another version of the result that relates symplectically reduced manifolds and coadjoint orbits, in the case of a non-equivariant momentum map,
but without passing to the corresponding affine action.
If $\J_R:S\to \g^*$ is non-equivariant, one can always find a central extension $\hat \g$ of $\g$ such that $\J_R$ extends to a 
$\hat\g$-equivariant momentum map $\J_R:S\to\hat\g^*$.
We shall make a weaker assumption than the existence of a  Lie group extension of $G$
that integrates $\hat\g$, namely  we shall assume the existence of  a $G$-action on $\hat\g^*$
that integrates the coadjoint action of $\g$ on $\hat\g^*$
(a replacement for the coadjoint action in the extended group).
In the special case when $\g$ is a perfect Lie algebra,
this is equivalent to the $G$-equivariance of the Lie algebra cocycle on $\g$
that describes the central extension $\hat\g$ \cite{Neeb}.
In the particular case where the  Lie algebra extension $\hat\g$ integrates to a Lie group extension $\hat G$ of $G$ and $\J_R$ is equivariant,
we can perform symplectic reduction at $\si\in\hat\g^*$
for the $\hat G$-action.
Because of the non-connectedness of $G$, the extension $\hat G$
might be an abelian extension that is non-central (of course the pullback to the identity component of $G$ is still a central extension).
Still the isotropy group of $\si$ is a subgroup of $G$
and the ordinary reduced symplectic manifold  at $\si$ 
is the quotient of $\J_R^{-1}(\si)$ by its action.

One can proceed similarly even if no extension $\hat G$ exists.
Let $G_\si$ be the stabilizer of $\si\in\hat\g^*$
for the $G$-action that integrates the coadjoint action of $\g$ on $\hat\g^*$.
If $\J_R$ is $G$-equivariant for the action above,
than the quotient $S_\si= \J_R^{-1}(\si)/G_\si$ 
can be endowed with a symplectic form $\om_\si$ 
whose pullback to the level set of $\J_R$
is the restriction of the symplectic form $\om$.
In the particular case where the Lie algebra extension $\hat\g$ integrates to a Lie group extension $\hat G$ of $G$, then $S_\si$
is the ordinary reduced symplectic manifold  at $\si$ for the $\hat G$-action.

\color{black}
\begin{proposition}\label{id2}
Let $(S,\om)$ be a symplectic manifold with commuting mutually orthogonal symplectic actions of $G$ and $H$.
On one hand we assume that $H$ is connected with 
equivariant momentum map  $\J_L$. 
On the other hand we assume that, although  the $G$-action does not admit 
an equivariant momentum map, there is a 
$\hat\g$-equivariant momentum map $\J_R:S\to\hat\g^*$ for the induced action 
of a fixed central extension $\hat\g$ of $\g$. 
Moreover we assume that, although $\hat\g$ might not be integrable to
a central extension of the (possibly non-connected) Lie group $G$, 
the coadjoint action of $\g$ on $\hat\g^*$ integrates to a $G$-action on $\hat\g^*$ that makes $\J_R$ a $G$-equivariant map.

If $G$ acts transitively on level sets of $\J_L$
and $H$ acts transitively on connected components of level sets of ${\J}_R$,
then the connected components of $S_\si=\J_R^{-1}(\si)/G_\si$
are symplectically diffeomorphic to coadjoint orbits of $H$. 
\end{proposition} 
\color{black}

The proof is similar to that of Proposition \ref{id1}.


\section{The ideal fluid dual pair}\label{s5}

A pair of momentum maps related to the dynamics of Euler equations of a perfect fluid was presented in \cite{Marsden-Weinstein}. The construction goes as follows.
Let $(S,\omega)$ be a symplectic manifold and $(M,\mu)$ be a compact manifold
endowed with volume form $\mu$. The manifold $C^\infty(M,S)$ of all smooth maps from $M$ to $S$ is naturally endowed with a symplectic form $\bar\om$ given by
\begin{equation}\label{crc}
\bar\om_f(u_f,v_f)=\int_M\om(u_f,v_f)\mu,
\end{equation}
where $u_f$ and $v_f$ are vector fields on $S$ along $f\in C^\infty(M,S)$. The left action of the group $\operatorname{Diff}_{\rm symp}(S)$ of symplectic diffeomorphisms and the right action of the group $\operatorname{Diff}_{\rm vol}(M)$ of volume preserving diffeomorphisms are two commuting symplectic actions. As explained \cite{Marsden-Weinstein}, to these actions is {formally} associated a pair of momentum maps
\begin{equation}\label{MW_dual_pair} 
\X_{\rm symp}(S )^*\stackrel{\J_L}{\longleftarrow} C^\infty(M,S)\stackrel{\J_R}{\longrightarrow}\X_{\rm vol}(M)^*.
\end{equation} 
The right momentum map $ \mathbf{J} _R(f)=-f^*\om $ represents Clebsch variables for the Euler equations seen as a Hamiltonian system on $\X_{\rm vol}(M)^*$. The left momentum map $ \mathbf{J} _L(f)=f_*\mu $ is a
constant of motion for the induced Hamiltonian system on $C^\infty(M, S)$. It also describes point vortex solutions of the two dimensional Euler equations on $S$ when $\operatorname{dim}(M)=0$ and $S$ is two dimensional.

A rigorous study of \eqref{MW_dual_pair} has been carried out in \cite{GBVi2011}. In particular, it was shown that, 
in order to verify the momentum map conditions, one needs to restrict the action to the subgroups $ \operatorname{Diff}_{\rm ex}(M)\subset \operatorname{Diff}_{\rm vol}(M) $ and $ \operatorname{Diff}_{\rm ham}(S )\subset \operatorname{Diff}_{\rm symp}(S) $ of exact volume preserving and Hamiltonian diffeomorphisms, 
respectively (for this purpose a calculus on manifolds of functions
has been developed in \cite{Vizman}). Then, in order to obtain the equivariance property of the momentum maps, it is necessary to consider the central extension $H=\Diff_{\quant}(P)_0$ of $ \operatorname{Diff}_{\rm ham}(S )$ (when $ \omega $ is prequantizable) and the Ismagilov central extension $G=\widehat{\operatorname{Diff}}_{\rm ex}(M)$ of $ \operatorname{Diff}_{\rm ex}(M)$ (see Section \ref{s1}). 
Finally, to 
show the dual pair property, one needs to restrict to the open subset $ \operatorname{Emb}(M, S)$ of all embeddings from $M$ into $S$ and assume $H^1(M)=0$ (in which case $ \operatorname{Diff}_{\rm ex}(M)$ coincides with the identity component $ \operatorname{Diff}_{\rm vol}(M)_0$). 
The resulting dual pair of Poisson maps is
\begin{equation}\label{MW_GBVi_dual_pair} 
\Den_c(S)=C^\infty(S)^\ast \stackrel{\J_L}{\longleftarrow} \operatorname{Emb}(M,S)\stackrel{\J_R}{\longrightarrow} \left( \Omega ^{k-2}(M)/ \mathbf{d} \Omega ^{k-3}(M) \right)  ^\ast = Z^2(M),
\end{equation} 
where $ \mathbf{J} _L(f)=f _\ast \mu $ and $ \mathbf{J} _R( f)=-f ^\ast \omega $, see Theorem 4.5 in \cite{GBVi2011}. Recall that both the definitions of the momentum maps and of the dual pair property only depend on the infinitesimal actions.

In fact, the actions verify the stronger property to be mutually completely orthogonal:
\[
\X_{\ex}(M)_{\Emb}^{\ \bar\om}=\X_{\ham}(S)_{\Emb},\quad
\X_{\ham}(S)_{\Emb}^{\ \bar\om}=\X_{\ex}(M)_{\Emb}.
\] 
In other words, when $H ^1 (M)=0$, the group $\operatorname{Diff}_{\rm ex}(M)$
(hence $G= \widehat{\operatorname{Diff}}_{\rm ex}(M)$ too) acts \textit{infinitesimally transitively} on the level sets of $ \mathbf{J} _L$ and  $\operatorname{Diff}_{\rm ham}(S)$ (hence $H=\Diff_{\quant}(P)_0$ too) acts \textit{infinitesimally transitively} on the level sets of $ \mathbf{J} _R$. As these groups are connected and all the constructions in the proof can be performed smoothly depending on a parameter, these actions are \textit{transitive} on the \textit{connected components} of the level sets
of momentum maps.
It is proven in Lemma 3 \cite{GBVi2012} that, without the condition $H^1(M)=0$,
$ \operatorname{Diff}_{\rm vol}(M)$ acts transitively on the level sets of $ \mathbf{J} _L$ (not just on its connected components). We recall below the proofs of the transitivity results we shall use later.

\begin{lemma}[\cite{GBVi2012}]\label{lem1} The group $\Diff_{\rm vol}(M)$ acts transitively on the level sets of the momentum map $\J_L:\Emb(M,S)\to
\Den_c(S)$.
\end{lemma}
\begin{proof}
Let $f_1,f_2\in\Emb(M,S)$  such that $\J_L(f_1)=\J_L(f_2)$, hence
$\int_M(h\o f_1)\mu=\int_M(h\o f_2)\mu$ for all $h\in C^\oo(S)$.
A first consequence is that the two embeddings have the same image in $S$,
so there exists $\ps\in\Diff(M)$ such that $f_2=f_2\o\ps$.
We rewrite the identity above as $\int_M(h\o f_2)\ps^*\mu=\int_M(h\o f_2)\mu$
for all $h\o f_2\in C^\oo(M)$, and we deduce $\ps^*\mu=\mu$.
Thus we found $\ps\in\Diff_{\vol}(M)$ such that $f_2=f_1\o\ps$.
\end{proof}

\begin{lemma}[\cite{GBVi2011}]\label{lem2}
If $H^1(M)=0$, then the group $\Diff_{\ham}(S)$ acts transitively {on each of} the connected components of level sets of the momentum map $\J_R:\Emb(M,S)\to Z^2(M)$.
\end{lemma}
\begin{proof}
First we show that the action of the Lie algebra $\X_{\ham}(S)$ on level sets of $\J_R$, given by $(X_h)_{\Emb}(f)=X_h\o f$,
is infinitesimally transitive. This means to show that, 
for any $v\o f\in T_f\Emb(M,S)$, where $v\in\X(S)$ such that
$T_f\J_R(v\o f)=-f^*\pounds _{v}\om=0$,
there is a Hamiltonian vector field $X_h$ on $S$ with $v\o f=X_h\o f$.
Indeed, the 1-form $f^*\mathbf{i} _{v}\om$ on $M$ is closed, hence exact, since $H^1(M)=0$. There exists  $h_1\in C^\oo(S)$ such that $f^*\mathbf{i} _{v}\om
=\dd(h_1\o f)$, since $f$ is an embedding.
Now the 1-form $\be_f$ on $S$ along $M$ defined by
\[
\be_f:=\mathbf{i} _{v\o f}\om-\mathbf{d} h_1\o f \in\Ga(f^*T^*S)
\] 
vanishes on vectors tangent to $f(M)\subset S$. 
By Lemma 4.2 in \cite{GBVi2011}, that is detached from the proof of Proposition 3 in \cite{Haller-Vizman}, we find $h_2\in C^\oo(S)$ such that $\be_f=\mathbf{d} h_2\o f$.
It follows that $\mathbf{d} (h_1+h_2)\o f=\mathbf{i} _{v\o f}\om$,
so $v\o f=X_{h_1+h_2}\o f$.

As in the proof of Proposition 2 in \cite{Haller-Vizman}
we remark that all the constructions above can be performed smoothly depending on a parameter,
so $\Diff_{\ham}(S)$ acts transitively on connected components of the level sets of the momentum map $\J_R$.
\end{proof}
\medskip

If $S$ is compact, then the central extension of $ \operatorname{Diff}_{\rm ham}(S)$ is not needed and we have the dual pair 
\begin{equation}\label{MW_GBVi_dual_pair_compact} 
\mathfrak{X}  _{\rm ham}(S)^\ast \stackrel{\J^0 _L}{\longleftarrow} \operatorname{Emb}(M,S)\stackrel{\J_R}{\longrightarrow} Z^2(M),
\end{equation}
where $\left\langle\mathbf{J} ^0_L(f), X_h \right\rangle =\int_M( h_0 \circ f) \mu $, with $h_0$ the unique Hamiltonian function of $X_h$ with zero integral on $S$.


\section{Coadjoint orbits as nonlinear Grassmannians}\label{s6}

The goal of this section is to carry out symplectic reduction for the right momentum map  of the ideal fluid dual pairs \eqref{MW_GBVi_dual_pair}, \resp \eqref{MW_GBVi_dual_pair_compact}, in order to describe coadjoint orbits of the group of Hamiltonian diffeomorphisms in terms of nonlinear Grassmannians. The case of zero momentum allows us to obtain Grassmannians of isotropic volume submanifolds as coadjoint orbits, slightly generalizing the results in \cite{W90} and \cite{L}. At the other extreme, the case of a nondegenerate momentum recovers the identification of connected components of the nonlinear symplectic Grassmannian with coadjoint orbits, thereby recovering the result of \cite{Haller-Vizman}.
Some comments about the intermediate cases will be given. Our approach naturally yields concrete expressions for the tangent spaces to the orbits and hence, explicit formulas for the orbit symplectic form.

\paragraph{Coadjoint orbits of the Hamiltonian group.}
Every compact $2m$-dimensional symplectic manifold $(S,\om)$ is a coadjoint orbit of its 
group of Hamiltonian diffeomorphisms $\Diff_{\ham}(S)$. The identification is made via the momentum map
\[
\J:S\hookrightarrow\X_{\ham}(S)^*,\quad\langle\J(s),X_h\rangle=h_0 (s),
\]
where $h_0 \in C^\oo(S)$ is the unique Hamiltonian function for $X_h$ with zero integral $\int_S h_0 \om^m=0$.

Another class of infinite dimensional coadjoint orbits of $\Diff_{\ham}(S)$ has been  described in \cite{Haller-Vizman}. \textcolor{black}{ Let $M$ be a compact $2n$-dimensional manifold. The connected components of the nonlinear symplectic
Grassmannian $\Gr_{\symp}^M(S)$ of symplectic submanifolds of $S$ of type $M$ are coadjoint orbits of $\Diff_{\ham}(S)$}. The identification is done via the momentum map associated to the natural Hamiltonian action of $ \operatorname{Diff}_{\rm ham}(S)$ on an arbitrary connected component $\Gr_{\symp}^M(S)_0$, endowed with the symplectic form naturally induced by  $ \frac{1}{n+1}\omega ^{n+1} \in \Omega ^{2(n+1)}(S)$. This momentum map is given by
\[
\J:\Gr_{\symp}^M(S)_0\hookrightarrow\X_{\ham}(S)^*,\quad\langle\J(N),X_h\rangle=\int_Nh_0 \om^n.
\]
If $S$ is not compact and $ \omega $ is prequantizable, then $\Gr_{\symp}^M(S)_0$ is a coadjoint orbit of the quantomorphism group
via the identification
\begin{equation}\label{stefan}
\J:\Gr_{\symp}^M(S)_0\hookrightarrow C^\oo(S)^*,\quad\langle\J(N),h\rangle=\int_Nh \om^n.
\end{equation}

In \cite{W90}, it was heuristically shown that the leaves of a certain foliation (the isodrastic foliation) of the space of weighted Lagrangian submanifold of $S$ can be identified with coadjoint orbits of the group of Hamiltonian diffeomorphisms of $S$. This fact was rigorously shown in \cite{L} and generalized to the leaves of a foliation of the space of weighted isotropic submanifolds of $S$.
Here a weight is a volume form of total measure 1.

\medskip

We shall obtain below those coadjoint orbits containing isotropic submanifolds 
\textcolor{black}{$N\subset S$ with $H^1(N)=0$}, as well as new ones, in a systematic way by using the general results in Propositions \ref{id1} and \ref{id2} applied to the ideal fluid dual pair as formulated in \cite{GBVi2011}. In the case of coadjoint orbits studied in \cite{L}, our approach yields much concrete expressions for the tangent spaces to the coadjoint orbits and hence, explicit formulas for the orbit symplectic form.
Our setting also allows us to build a prequantum bundle for the isotropic case (in \S\ref{s7}), which recovers the Berry bundle of \cite{W90} in the Lagrangian case.

More precisely, we will show below that, {if $(S, \omega )$ is prequantizable}, the connected components of the \textit{nonlinear Grassmannian of isotropic volume submanifolds} of $S$ of type $(M,\mu)$ are coadjoint orbits of $\Diff_{\quant}(P)_0$. If $S$ is compact, then these are coadjoint orbits of $ \operatorname{Diff}_{\rm ham}(S)$. If the symplectic manifold $S$ is exact,
the connected components of the nonlinear Grassmannian 
of \textit{exact} isotropic volume submanifolds of $S$ of type $(M,\mu)$
are coadjoint orbits of $\Diff_{\quant}(P)_0$.

\paragraph{Isotropic embeddings.} 
Recall that the dual pair \eqref{MW_GBVi_dual_pair}
consists of momentum maps for the commuting mutually orthogonal actions 
of $G=\widehat{\operatorname{Diff}}_{\rm ex}(M)$ and $H=\Diff_{\quant}(P)_0$ on $\Emb(M,S)$,
where $(M,\mu)$ is a compact $k$-dimensional manifold
endowed with volume form $\mu$ with $H^1(M)=0$
and $(S,\om)$ a prequantizable symplectic $2m$-dimensional manifold.

The zero level set of the momentum map $\J_R: \operatorname{Emb}(M, S) \rightarrow Z ^2 (M)$, $ \mathbf{J} _R  (f)=f^*\om$ is the manifold of \textit{isotropic embeddings}
\[
\J_R^{-1}(0)=\{f\in\Emb(M,S): f^*\om=0\}=:\Emb_{\iso}(M,S).
\]
The dual pair property of \eqref{MW_GBVi_dual_pair} 
ensures that the tangent space 
\begin{equation}\label{isos}
T_f\Emb_{\iso}(M,S)=\ker T_f\J_R=(\ker T_f\J_L)^{\bar\om}=(\X_{\ham}(S))_{\Emb(M,S)}(f)
\end{equation}
is the space of infinitesimal generators at the isotropic embedding $f$ for the left action of Hamiltonian vector fields.
 
\paragraph{Isotropic nonlinear Grassmannian.}
Recall that the momentum map $\J_R$ is associated to the action
of the central extension $\widehat{ \operatorname{Diff}}_{\rm ex}(M)$ of the group of exact volume preserving diffeomorphisms
$\Diff_{\ex}(M)=\Diff_{\rm vol}(M)_0$. Unfortunately, symplectic reduction at zero relative to this group, $ \mathbf{J} _R ^{-1} (0)/ \widehat{ \operatorname{Diff}}_{\rm ex}(M)= \J_R ^{-1} (0)/\operatorname{Diff}_{\rm ex}(M)$,
leads to a covering space of a nonlinear Grassmannian of volume submanifolds.
Moreover, the hypotheses of Proposition \ref{till} are not satisfied: ${\operatorname{Diff}}_{\rm ex}(M)$  does not act transitively on the level sets of $ \mathbf{J} _L $,
hence  this symplectic reduced space might not be isomorphic to a coadjoint orbit.

We shall instead consider the quotient space $\J_R ^{-1} (0)/\operatorname{Diff}_{\rm vol}(M)$. Remarkably, using Proposition \ref{id2}, we will show that this is still a symplectic reduced space, despite the fact that the central extension
of the non-connected group $\Diff_{\vol}(M)$ is not known. 
 
The symplectic reduced space ${\mathbf{J}} _R ^{-1} (0)/ \operatorname{Diff}_{\rm vol}(M)$ is given by the image of the manifold ${\mathbf{J}} _R^{-1}(0)= \operatorname{Emb}_{\iso}(M,S)$ by the map $\pi^\mu$ from \eqref{pim}. It is therefore the manifold
\[
\Gr_{\iso}^{M,\mu}(S)=\left\{(N,\nu):N\in\Gr_{\iso}^M(S),\;\nu\in\Vol(N),\;\int_N\nu=\int_M\mu\right\} 
\] 
of \textit{isotropic volume submanifolds of $S$ of type $(M,\mu)$}. In the appendix  is shown that $\Gr_{\iso}^{M,\mu}(S)$ is a smooth manifold, the base of the principal bundle 
$\pi^\mu:\Emb_{\iso}(M,S) \rightarrow \Gr_{\iso}^{M,\mu}(S)$ with structure group $\Diff_{\rm vol}(M)$ (see also \cite{L}).

In order to describe the reduced symplectic form on $\Gr_{\iso}^{M,\mu}(S)$, we need a concrete realization for the tangent space $T_{(N,\nu)}\Gr_{\iso}^{M,\mu}(S)$. We first consider the tangent space to the manifold $\Gr_{\iso}^M(S)=\operatorname{Emb}_{\iso}(M,S)/ \operatorname{Diff}_{\rm vol}(M)$ of \textit{isotropic submanifolds of type $M$}.

\paragraph{Tangent space to $\Gr_{\iso}^M(S)$.}
The Lie algebra $\X_{\ham}(S)$ acts transitively on  $\Gr_{\iso}^M(S)$, so
the tangent space at $N$, as a subspace of $T_N\Gr^M(S)=\Ga(TN^\perp)$, can be written as
\begin{equation}\label{tgt_space_Gr_iso_S}
T_N\Gr_{\iso}^M(S)\stackrel{\eqref{isos}}{=}\{T_f\pi(v\o f):v\in\X_{\ham}(S)\}\stackrel{\eqref{bigp}}{=}
\{\textcolor{black}{[X_h|_N]}\in\Ga(TN^\perp):h\in C^\oo(S)\},
\end{equation} 
where $TN^\perp= (TS|_{N})/TN$ and $f$ is any isotropic embedding with $f(M)=N$.

In the special case $\dim M=\frac12\dim S=m$ we get the nonlinear Grassmannian $\Gr_{\Lag}^M(S)$ of \textit{Lagrangian submanifolds of $S$ of type $M$}.
Because for any Lagrangian submanifold $L$ of $ S$ the symplectic orthogonal $TL^\om$ coincides with $TL$,
the symplectic form $\om$ defines an isomorphism between
$TL^\perp=(TS|_L)/TL$ and $T^*L$, so
$T_L\Gr^M(S)=\Ga(TL^\perp)\cong\Om^1(L)$ and
\begin{equation}\label{lagri}
T_L\Gr^M_{\Lag}(S)=\{\textcolor{black}{[X_h|_L]}\in\Ga(TL^\perp):h\in C^\oo(S)\}
\cong\dd C^\oo(L).
\end{equation}
All the Lagrangian submanifolds $L\in\Gr^M_{\Lag}(S)$, being of type $M$, 
satisfy $H^1(L)=0$, hence this description  of the tangent space agrees with  the general result 
$T_L\Gr^M_{\Lag}(S)=Z^1(L)$ \cite{W71}.

A description similar to \eqref{lagri} 
\textcolor{black}{in the isotropic case}
is obtained using  \eqref{tgt_space_Gr_iso_S}: 
\begin{equation}\label{ting}
T_N\Gr_{\iso}^M(S)=\{[v_N]\in\Ga(TN^{\perp}):(\mathbf{i} _{v_N}\om)|_{TN}=\mathbf{d} h_0,h_0\in C^\oo(N)\}.
\end{equation}
To show this, we assume that $v_N\in\Ga(TS|_N)$ is such that $(\mathbf{i} _{v_N}\om)|_{TN}=\mathbf{d} h_0$ for $h_0\in C^\oo(N)$. Let $h_1\in C^\oo(S)$ be an extension of $h_0$ to $S$. The section $\al:=\mathbf{i} _{v_N}\om-(\mathbf{d} h_1)|_N\in\Ga(T^*S|_N)$ verifies $\al|_{TN}=0$, therefore, by Lemma 4.2 in \cite{GBVi2011}, there exists $h_2\in C^\oo(S)$
such that  $\al=(\mathbf{d} h_2)|_N$. 
Then $ \mathbf{i} _{v_N}\om=(dh)|_N$ for  $h=h_1+h_2$, so $v_N=
X_h|_N$
and $[v_N]=\textcolor{black}{[X_h|_N]}$. 
{The other inclusion is readily verified.}

\paragraph{Reduced symplectic form.}
 With a choice of a  Riemannian metric $g$ on $S$,
the vector bundle $TN^\perp$ is isomorphic to the orthogonal complement 
$TN^{\perp_g}$ of $TN$ in $TS|_N$ relative to the metric $g$.
As shown in \cite{GBVi2013}, the differential of the projection $\pi^\mu$ in \eqref{pim} can be written, for all $v\in\X_{\ham}(S)$, as 
\begin{equation}\label{tunu}
T_f\pi^\mu:  v\o f \in T_f\Emb_{\iso}(M,S)\;\longmapsto\; (v|_N^\per,\pounds _{v|_N^{\parallel_g}}\nu)\in
T_{(N,\nu)}\Gr_{\iso}^{M,\mu}(S).
\end{equation}
\color{black}
The tangent space to $\Gr_{\iso}^{M,\mu}(S)$ can be identified with
\begin{align*}
T_{(N,\nu)}\Gr_{\iso}^{M,\mu}(S)
&=T_N\Gr_{\iso}^M(S)\x\dd\Om^{k-1}(N).
\end{align*}
To show this we consider the $\Diff(M)$-action on
$\Emb_{\iso}(M,S)$ by composition on the right. The infinitesimal generator of $w\in\X(M)$
at the isotropic embedding $f$ is given by $Tf\o w\in T_f\Emb_{\iso}(M,S)$. It is mapped by \eqref{tunu}
to the pair $\big(0,(_N|f)_*(\pounds_w\mu)\big)$, where $_N|f$ is the diffeomorphism $_N|f: M \rightarrow N$. Tangent vectors of this form clearly span $\{0\}\x\dd\Om^{k-1}(N)$.
\color{black}

The following proposition gives the concrete expression of the reduced symplectic form for the description of the tangent space $T_{(N,\nu)}\Gr_{\iso}^{M,\mu}(S)$ given above. 

\begin{theorem}\label{formula_KKS} 
The reduced symplectic form 
on  $\Gr_{\iso}^{M,\mu}(S)$ 
is
\begin{align}\label{trei}
(\om_0)_{(N,\nu)}\big((u_N,\mathbf{d} \gamma),(v_N,\mathbf{d} \la)\big)&=\int_N\om(u_N,v_N)\nu
+\int_N(\mathbf{i} _{u_N}\om)|_{TN}\wedge\la-\int_N(\mathbf{i} _{v_N}\om)|_{TN}\wedge\gamma\nonumber\\
&=\int_N\big(  \om(u_N,v_N)\nu + h_{v_N} \mathbf{d} \gamma - h_{u_N} \mathbf{d} \lambda \big),
\end{align} 
for all $u_N,v_N\in T_N\Gr_{\iso}^M(S)$, $\gamma,\la\in\Om^{k-1}(N)$, where $h_{ u_N}$ and $h_{ v_N}$ are  functions on $N$ associated 
to $u_N$ and $v_N$ as in \eqref{ting}. 
\end{theorem}

\begin{proof}
First we remark the independence of this expression on the choice of the potential forms $\gamma $ and $\la$,
because $(\mathbf{i} _{u_N}\om)|_{TN}$ and $(\mathbf{i} _{v_N}\om)|_{TN}$ are exact 1-forms by \eqref{ting}.

Now we check that $\om_0$ defined with \eqref{trei} is indeed the reduced symplectic form.
Let $f\in\Emb_{\iso}(M,S)$ such that $f(M)=N$ and $f_*\mu=\nu$, 
and let $u,v\in\X_{\ham}(S)$ such that $v_N=v|_N^\per$ and $\la=\mathbf{i} _{v|_N^{\parallel_g}}\nu$ \resp $u_N=u|_N^\per$ and $\ga=\mathbf{i} _{u|_N^{\parallel _g }}\nu$. 
We have 
\begin{align*}
(\om_0)_{(N,\nu)}&\big((u_N,\mathbf{d} \gamma),(v_N,\mathbf{d} \la)\big)=(\om_0)_{\pi^\mu(f)}\big(T_f\pi^\mu(u\o f),T_f\pi^\mu(v\o f)\big)\\
&\stackrel{\eqref{omega_0}}{=}\langle\J_L(f),\om(u,v)\rangle
{=}\int_Mf^*\om(u,v)\mu=\int_N\om(u|_N,v|_N)\nu\\
&\stackrel{\ \ \ \ }{=}\int_N\om(u|_N^\per,v|_N^\per)\nu
+\int_N(\mathbf{i} _{u|_N^\per}\om)|_{TN}\wedge \mathbf{i} _{v|_N^{\parallel_g}}\nu-\int_N(\mathbf{i} _{v|_N^\per}\om)|_{TN}\wedge \mathbf{i} _{u|_N^{\parallel_g}}\nu\\
&\stackrel{\eqref{tunu}}{=}\int_N\om(u_N,v_N)\nu
+\int_N(\mathbf{i} _{u_N}\om)|_{TN}\wedge\la-\int_N(\mathbf{i} _{v_N}\om)|_{TN}\wedge\gamma\nonumber\\
&\stackrel{\ \ \ \ }{=}\int_N\big(  \om(u_N,v_N)\nu + h_{v_N} \mathbf{d} \gamma - h_{u_N} \mathbf{d} \lambda \big),
\end{align*}
using at step four the identity $\om(u|_N^{\parallel_g},v|_N^{\parallel_g})=0$, 
which follows from the fact that $N$ is an isotropic submanifold of $(S,\om)$.
\end{proof}

\paragraph{The Lagrangian case.} 
The tangent space to $\Gr_{\Lag}^{M,\mu}(S)$,
the nonlinear Grassmannian of \textit{Lagrangian volume submanifolds $(L,\nu)$
of type $(M,\mu)$}, becomes:
\begin{equation}\label{didi}
T_{(L,\nu)}\Gr_{\Lag}^{M,\mu}(S)
=T_L\Gr_{\Lag}^M(S)\x\dd\Om^{m-1}(L)\stackrel{\eqref{lagri}}{=}\dd C^\oo(L)\x\dd\Om^{m-1}(L).
\end{equation}
Here $\dd\Om^{m-1}(L)$ can be identified with the (regular) dual of $\dd C^\oo(L)$ through the pairing
\begin{equation}\label{see}
(\dd h,\dd\la)=-\int_L h\dd\la=\int_L \dd h \wedge \la.
\end{equation}
It is always possible to choose a Riemannian metric $g$ on $S$ that is compatible with $\om$, see, e.g., Proposition 4.1 in \cite{MDSa98}.
This means there exists an almost complex structure $J$ such that
$\om(v,Jw)=g(v,w)$ for all $v,w\in T_sS$.
Then the reduced symplectic form \eqref{trei}
on $\Gr_{\Lag}^{M,\mu}(S)$ 
with tangent space \eqref{didi} becomes
\begin{align}\label{patru}
(\om_0)_{(L,\nu)}\big((\dd h_1,\mathbf{d} \la_1),(\dd h_2,\mathbf{d} \la_2)\big)
&=\int_L\om(u_{1L},u_{2L})\nu+\int_L\dd h_1\wedge\la_2-\int_L\dd h_2\wedge\la_1\nonumber\\
&=(\dd h_1,\dd\la_2)-(\dd h_2,\dd\la_1),
\end{align} 
for all $(\dd h_1,\mathbf{d} \la_1),(\dd h_2,\mathbf{d} \la_2)\in \dd C^\oo(L)\x\dd\Om^{m-1}(L)$.
Here $u_{1L},u_{2L}\in \Ga(TL^\per)$ are uniquely determined by 
$(\ii_{u_L}\om)|_{TL}=\dd h$.
The first term 
vanishes because, for $s\in L$, the subspace $(T_sL)^\per=J(T_sL)\subset T_sS$ is Lagrangian 
whenever $T_sL\subset T_sS$ is Lagrangian.

\color{black}
The space $\dd\Om^{m-1}(L)$ can be identified with the regular dual of $\dd C^\oo(L)$, as we have seen in \eqref{see}, thus the orbit symplectic form \eqref{patru} is similar
to a canonical cotangent bundle symplectic form. 
This similarity comes from the fact that $\om_0$ can be obtained by symplectic reduction
on the cotangent bundle $T^*\Gr_{\Lag}^M(S)$ \cite{W90},
as explained also at the end of Section \ref{s7}.
\color{black}

\paragraph{Isotropic nonlinear Grassmannians as coadjoint orbits.} {On the space of weighted Lagrangian submanifolds, the distribution corresponding to the subspace $ \mathbf{d} C^\infty(L) \subset Z ^1 (L)$ of exact 1-forms integrates into a foliation, called the isodrastic foliation. It was argued in \cite{W90} that the leaves of this foliation, the isodrasts, are coadjoint orbits of $\Diff_{\ham}(S)$. This result was generalized in \cite{L} to the case of isotropic submanifolds.}

If $H^1(M)=0$, the distribution coincides with the whole tangent space and the leaves are connected components of  the space of
weighted isotropic submanifolds.
We shall obtain them by symplectic reduction in the ideal fluid dual pair,
so their coadjoint orbit property follows from the general results 
about dual pairs of momentum maps for mutually completely orthogonal actions.

\color{black}
\begin{theorem}\label{unul}
Let $(S,\om)$ be a prequantizable symplectic manifold 
and  $(M,\mu)$ be a  compact manifold with volume form $\mu$ and $H^1(M)=0$.
Any connected component $\O$
of  $\Gr_{\iso}^{M,\mu}(S)$,
endowed with the symplectic form \eqref{trei},
is symplectically diffeomorphic to a coadjoint orbit  of the identity component $\Diff_{\quant}(P)_0$. The symplectic diffeomorphism reads
\[
\bar \J_L:\O\rightarrow \operatorname{Den}_c(S), \quad 
{\langle \bar\J_L(N, \nu ),h\rangle= \int _N h|_N\nu}
\]
and is a momentum map for the action of $\Diff_{\quant}(P)_0$ on  $\O\subset \Gr_{\iso}^{M,\mu}(S)$.
\end{theorem}
\color{black}
\begin{proof} 
 The result follows from Proposition \ref{id2} applied to the ideal fluid dual pair \eqref{MW_GBVi_dual_pair}, with the groups $H=\Diff_{\quant}(P)_0$ and $G= \operatorname{Diff}_{\rm vol}(M)$ 
(not just its identity component $ \operatorname{Diff}_{\rm ex}(M)$). Indeed, from \cite{GBVi2011} we know that the Lie algebra actions associated to the dual pair \eqref{MW_GBVi_dual_pair} are completely mutually orthogonal; from Lemma \ref{lem1}, the action of $G= \operatorname{Diff}_{\rm vol}(M )$ is transitive on the level sets of  $\J_L(f)=f_*\mu$, and $ \mathbf{J} _L$ is $H$-equivariant.
From Lemma \ref{lem2}, the action of $H=\Diff_{\quant}(P)_0$ (connected) is transitive on each connected component of the level sets of ${\mathbf{J}} _R(f)=-f^*\om$, and $\J_R$ is $G$-equivariant for the $G$-action on 
$Z^2(M)=\hat\g^*$ given by $\ph\cdot\si=\ph^*\si$. 
It remains to verify that this action integrates the infinitesimal coadjoint action in $\hat\g$:
\begin{align*}
\langle \widehat\ad^*_\xi\si,[\be]\rangle
=\langle\si,\{[\al],[\be]\}\rangle
\stackrel{\eqref{iimu}}{=}\langle\si,[\pounds_\xi\be]\rangle
=\langle\pounds_\xi\si,[\be]\rangle,
\end{align*}
for all $\xi=X_\al\in\g=\X_{\vol}(M)=\X_{\ex}(M)$, $[\al]\in\hat\g=\Om^{k-2}(M)/\dd\Om^{k-3}(M)$ and $\si\in\hat\g^*=Z^2(M)$.

The expression of $\bar\J_L$ follows from \eqref{jele}: for all $h\in C^\oo(S)$,
\[
\left\langle \bar\J_L(N, \nu ), h \right\rangle=\left\langle \bar\J_L(f(M), f_*\mu ), h \right\rangle =\left\langle \J_L(f), h \right\rangle =\int_M(h\o f)\mu
 =\int_Nh \nu,
\]
because $\J_L(f)=f_*\mu$.
\end{proof}

\color{black}
\begin{theorem}\label{doiu}
If $S$ is compact and $H^1(M)=0$, then any connected component $\O$ of  $\Gr_{\iso}^{M,\mu}(S)$,
endowed with the symplectic form \eqref{trei},
is symplectically diffeomorphic to a coadjoint orbit  of $ \operatorname{Diff}_{\rm ham}(S)$. The diffeomorphism is given by the momentum map
\[
\bar \J_L: \O\rightarrow  \mathfrak{X}_{\rm ham}(S ) ^\ast , \quad  \left\langle \bar\J_L(N, \nu ), X_h \right\rangle =\int_Nh _0 \nu ,
\]
where $h_0$ is the unique Hamiltonian function with zero integral on $S$
 associated to the Hamiltonian vector field $X_h$.
\end{theorem}
\color{black}
\begin{proof}  The result follows from Proposition \ref{id2} applied to the ideal fluid dual pair \eqref{MW_GBVi_dual_pair_compact}, with the groups $H= \operatorname{Diff}_{\rm ham}(S)$ and $G= \operatorname{Diff}_{\rm vol}(M)$.
\end{proof}

\rem{
\begin{proof}
By the proof of Proposition \ref{id2}, there is a Hamiltonian $\Diff_{\ham}(S)$-action on the reduced symplectic manifold
$\left( \Gr_{\iso}^{M,\mu}(S),\om_0 \right) $, namely $\ps\cdot(N,\nu)=(\ps(N),\ps_*\nu)$ for all $\ps\in\Diff_{\ham}(S)$.
An equivariant momentum map can thus be written 
for the Lie algebra action of the central extension $C^\oo(S)$ of ${\X}_{\ham}(S)$. 
It is given by  
$$
\bar\J_L:\Gr_{\iso}^{M,\mu}(S)\to C^\oo(S)^*,\quad \langle\bar\J_L(N,\nu),h\rangle=\int_N(h|_N)\nu, \quad h\in C^\oo(S).
$$
For compact $S$ one can use zero integral Hamiltonian functions, so the connected components
of $\Gr_{\iso}^{M,\mu}(S)$ are coadjoint orbits of $\Diff_{\ham}(S)$ too.
Note that we don't need the symplectic manifold $(S,\om)$ to be prequantizable.
\end{proof}}

\paragraph{Symplectic reduction at nonzero momentum.} Under the hypothesis $H^1(M)=0$, we shall perform the symplectic reduction from Proposition \ref{id2} at a nonzero element $-\si\in Z^2(M)$ \textcolor{black}{(\ie a presymplectic form on $M$).
The reduced symplectic space is 
$\J_R^{-1}({-\si})/\Diff_{\rm vol,\si}(M)$, for the stabilizer
\begin{equation}\label{muph}
\Diff_{\rm vol,\si}(M)=\{\ph\in\Diff(M):\ph^*\mu=\mu,\ph^*\si=\si\}.
\end{equation}
We show that the reduced space can be identified with the non-linear Grassmannian $\Gr^{M,\mu}_\si(S)$ of all volume submanifolds of type $(M,\mu)$ which,
when endowed with the pullback of the symplectic form $\om$, 
are diffeomorphic to the presymplectic volume manifold $(M,\mu,\si)$.
Indeed, }
\begin{align*}
&\{f\in\Emb(M,S):f^*\om=\si\}/\Diff_{\vol,\si}(M)\\
&\quad =\left\{(f(M),f_*\mu,f_*\si) : f \in\Emb(M,S), f^*\om=\si\right\}\\
&\quad =\left\{(N,\nu,i_N^*\om) : (N, \nu ) \in \Gr^{M,\mu}(S), \exists f\in\Emb(M,S)\text{ s.t. } f(M)=N, f_*\mu=\nu,f^*\om=\si\right\}\\
&\quad =\left\{(N,\nu) \in\Gr^{M,\mu}(S):\exists\ps\in\Diff(M,N)\text{ s.t. }\ps^*i_N^*\om=\si,\ps^*\nu=\mu\right\}\\
&\quad =\left\{(N,\nu) \in\Gr^{M,\mu}(S):(M,\mu,\si)\cong(N,\nu,i_N^*\om)\right\}=: \Gr^{M,\mu}_\si(S),
\end{align*}
where $i_N:N \rightarrow S$ denotes the inclusion.

Note that we have $\Gr^{M,\mu}_\si(S)=\J_R^{-1}(-\si)/\Diff_{\rm vol,\si}(M)$ and that $\Diff_{\rm vol,\si}(M)$ in \eqref{muph} might not be a Lie group for degenerate $\si$,
so the space $\Gr^{M,\mu}_\si(S)$ might fail to be a smooth manifold.
If it happens to be a smooth manifold, then by Proposition \ref{id2} applied to nonzero momentum
we deduce that the connected components of $\Gr^{M,\mu}_\si(S)$
are coadjoint orbits of the group of Hamiltonian diffeomorphisms, if $S$ is compact.

The reduced symplectic form $\om_{-\si}$  is formally given as in  \eqref{omega_0} by
\begin{equation}\label{omsi}
(\om_{-\si})_{\pi^{\mu}(f)}\left( T_f\pi^{\mu}(u\o f),T_f\pi^{\mu}(v\o f)\right) 
=\int_Mf^*\om(u,v)\mu
=\langle\J_L(f),\om(u,v)\rangle,
\end{equation}
where $u,v\in\X_{\ham}(S)$.


\paragraph{Nonlinear symplectic Grassmannians as coadjoint orbits.}
We now consider the special case when the presymplectic form $ \sigma $ is symplectic and $ \mu $ is the associated Liouville volume form, \ie $\mu=\si^n$, $2n= \operatorname{dim}M$. 
Then the reduced symplectic manifold $\Gr_\si^{M,\mu}(S)$
becomes a union of connected components of the nonlinear symplectic
Grassmannian {$\Gr_{\symp}^M(S)$}, namely those
components containing symplectic submanifolds of type $(M,\si)$.
Indeed,
$\Diff_{\vol,\si}(M)=\Diff_\si(M)$ since
$\ps_*\mu=\ps_*\si^n=\si^n=\mu$ for $\psi^*\si=\si$,
so the reduced symplectic space becomes
\begin{align*} 
\{f\in\Emb(M,S):f^*\om=\si\}/\Diff_{\si}(M)
&=\{(N,i_N^*\om):\exists\ps\in\Diff(N,M),\text{ s.t. }\ps^*i_N^*\om=\si\}\\
&=\left\{N\in\Gr^M(S) :(M,\si)\cong(N,i_N^*\om)\right\}\subset \Gr_{\symp}^M(S).
\end{align*}
With Proposition \ref{id2} we recover the result of \cite{Haller-Vizman}
that all the connected components of $\Gr_{\symp}^M(S)$
are coadjoint orbits of $\Diff_{\ham}(S)$.

In the general expression of the symplectic diffeomorphism \eqref{jele} we recognize the momentum map 
\eqref{stefan}:
\[
\langle\bar\J_L(N,\nu),h\rangle=\int_Nh\nu=\int_N hi_N^*\om^n=\int_N h\om^n,
\]
since the volume form on $N$ in this case is $\nu=\ps_*\mu=\ps_*\si^n=i_N^*\om^n$.
One can also compare the formula \eqref{omsi} for the orbit symplectic form
with the orbit symplectic form obtained in \cite{Haller-Vizman}.
At the point $\pi^{\mu}(f)=(f(M),f_*\mu)=(N,\nu=i_N^*\om^n)$,
for all  Hamiltonian vector fields $u,v$ on $S$  we have
\begin{align*}
(\om_{-\si})_{\pi^{\mu}(f)}\left( T_f\pi^{\mu}(u\o f),T_f\pi^{\mu}(v\o f)\right) 
&=\int_Mf^*\om(u,v)\mu=\int_Mf^*\om(u,v)f^*\om^n\\
&=\int_N\om(u,v)\om^n=\frac{1}{n+1}\int_N \mathbf{i} _v \mathbf{i} _u\om^{n+1}.
\end{align*}
Under the identification of  $\Gr_\si^{M,\mu}(S)$ with an open subset of
the nonlinear symplectic Grassmannian $\Gr_{\symp}^M(S)$,
this is exactly the orbit symplectic form obtained by \cite{Haller-Vizman}, as recalled at the beginning of Section \ref{s6}.


\section{Prequantizable coadjoint orbits}\label{s7}

Let $p:(P,\al)\to (S,\om)$ be a prequantum bundle over the symplectic manifold $S$. 
A horizontal submanifold of $P$ that covers a Lagrangian submanifold of $S$
is called \textit{Planckian} in \cite{Souriau}; a more familiar name is \textit{Legendre submanifold} of the contact manifold $(P,\al)$.
The prequantization of isodrasts of weighted Lagrangian submanifolds of $S$
is obtained in Section 4 of \cite{W90} using the symplectic reduction of the cotangent bundle of the space of Planckian submanifolds of $P$ 
that descend to Lagrangian submanifolds in an isodrast. This leads to the so called \textit{Berry bundle}.

{In this section we will generalize this result to the case of isotropic submanifolds, thereby obtaining a generalization of the Berry bundle.}
Given a prequantizable symplectic manifold $S$ 
with prequantum bundle  $p:(P,\al)\to(S,\om)$
and a compact manifold $(M,\mu)$ with integral total volume 
\begin{equation}\label{a}
a=\int_M\mu\in\ZZ,
\end{equation} 
we will build a prequantum bundle for the coadjoint orbits of the quantomorphism group that are
realized by connected components of $\Gr_{\iso}^{M,\mu}(S)$. In the last part of this Section, we show that our prequantum bundle recovers the Berry bundle in the Lagrangian case.

We first build a principal circle bundle over $\Emb_{\iso}(M,S)$. For this we define the \textit{manifold of horizontal embeddings},
\[
\Emb_{\hor}(M,P):=\{F\in\Emb(M,P):F^*\al=0\}.
\]
The  quantomorphism group $\Diff_{\quant}(P)$ acts on
 $\Emb_{\hor}(M,P)$ by composition on the right. The infinitesimal action of the Lie algebra of connection preserving vector fields
\[
\X_{\quant}(P){=}\{\xi_h:=X_h^{\hor}-(h\o p)E:h\in C^\oo(S)\}
\]
is given by $(\xi_h)_{\Emb}(F)=\xi_h\o F$.

\begin{lemma}\label{lemma_emb_hor} 
If the compact manifold $M$ satisfies $H^1(M)=0$,
then the principal circle action $(z,F)\mapsto{\rho_z \circ F}$
determines a principal circle bundle
\begin{equation}\label{pbar}
p_*: \Emb_{\hor}(M,P)\to\Emb_{\iso}(M,S), \quad p_*(F)=p\o F.
\end{equation}
\end{lemma}

\begin{proof} Every horizontal embedding $F$ descends to an isotropic embedding $f=p \circ F$,
since $f^*\om=
\mathbf{d} F^*\al=0$.
Two horizontal embeddings $F_1$ and $F_2$ that descend to the same isotropic embedding 
differ by the action of some element in $S^1$
because $F_2=\rh_g\o F_1$ for some function $g\in C^\oo(M,S^1)$,
but $0=F_2^*\al=F_1^*\al+g ^{-1} \mathbf{d} g=g ^{-1} \mathbf{d} g$ implies that $g$ is a constant function,
so $F_2=\rh_z\o F_1$, $z\in S^1$.

We show that each isotropic embedding $f$ of $M$ in $S$ 
can be lifted to a horizontal embedding $F$ of $M$ in $P$.
The restricted holonomy around loops in $f(M)$ is trivial because $f^*\om=0$.
Moreover, the holonomy groups are also trivial:
the image of the holonomy homomorphism $\pi_1(f(M))\to S^1$
is trivial because the abelianization of $\pi_1(f(M))$ is $H_1(f(M))=0$.
Hence, over $f(M)$, the circle bundle admits a flat trivialization,
which means it is foliated by horizontal leaves, and
each of them projects diffeomorphically onto $f(M)$.
The embedding of $M$ onto such a horizontal leaf,
obtained by composition of the inverse of such a diffeomorphism with $f$, 
is a horizontal embedding $F$ with $f=p\o F$.
\end{proof}

\medskip
We now define a principal connection on the circle bundle \eqref{pbar} whose curvature is 
the closed 2-form $\bar\om$ defined in \eqref{crc}.
The form $\al\in\Om^1(P)$, together with the volume form $\mu$ on $M$,
determines the form $\bar\al\in\Om^1(\Emb_{\hor}(M,P))$ given, 
for every vector field $v_F$ on $P$ along $F$, by 
\begin{equation}\label{alfa}
\bar\al(v_F):=\int_M\al(v_F)\mu.
\end{equation}

If the total volume of $M$ is equal to 1 (i.e., $a=1$ in \eqref{a}),
then $\bar\al$ is a principal connection 1-form with curvature $\bar\om$.
The verifications that $\bar\al$ is a principal connection with curvature $\bar\om$
can be done directly or by applying the calculus on manifolds of functions
developed in \cite{Vizman}, especially the {functorial} properties of the map $ \alpha \mapsto \bar \alpha $.
The $S^1$-invariance of $\bar\al$ follows from  $((\rh_z)_*)^*\bar\al
=\overline{\rh_z^*\al}=\bar\al$.
The form $\bar\al$ reproduces the infinitesimal generators because 
for the infinitesimal generator $E_{\Emb}$ of the $S^1$-action on $\Emb_{\hor}(M,P)$, which is induced by the infinitesimal generator $E\in\X(P)$ of the circle action on $P$, 
we have 
$$\bar\al(E_{\Emb})(F)=\bar\al(E\o F)=\int_M(\al(E)\o F)\mu=\int_M\mu=1,$$ since $\al(E)=1$.
Finally the curvature computation is $\mathbf{d} \bar\al=\overline{\mathbf{d} \al}=\overline{p^*\om}=(p_*)^*\bar\om$.

Later, we will use the following transitivity result.

\begin{lemma}\label{lem4}
If $H^1(M)=0$, then the Lie algebra $\X_{\quant}(P)$  acts transitively on  $\Emb_{\hor}(M,P)$.
\end{lemma}
\begin{proof}
We have  to show that, 
for any $v\o {F} \in T_F\Emb(M,P)$, where $v\in\X(P)$ is such that
$F^*\pounds _{v}\al=0$,
there is a connection preserving vector field $\xi_h$ on $P$ such that $v\o F=\xi_h\o F$.
Since $p$ induces a diffeomorphism between $F(M)$ and $f(M)$,
and we are interested only in the restriction of $v$ to $F(M)$,
we can assume without loss of generality that $v$ is projectable  to $\bar v\in\X(S)$. Hence there exists $\la\in C^\oo(P)$ such that
$v=\bar v^{\hor}-\la E$.
By differentiating the condition $F^*\pounds _{v}\al=0$ we get $f^*\pounds _{\bar v}\om=0$,
so,  by Lemma \ref{lem2},  there exists $h\in C^\oo(S)$ with $\bar v\o f=X_h\o f$. We compute
\[
0=F^*\pounds _{v}\al=F^*\mathbf{i}_{\bar v^{\hor}}\dd\al-F^*\dd\mathbf{i}_{\la E}\al=f^*\mathbf{i}_{\bar v}\om-F^*\dd\la=\dd(f^*h-F^*\la)
\]
so $F^*\la$ and $f^*h$ differ by a real constant $c$.
Finally we get
\[
v\o F=\bar v^{\hor}\o F-(F^*\la)(E\o F)=X_h^{\hor}\o F-(f^*(h+c))(E\o F)
=\xi_{h+c}\o F
\]
thus proving the transitivity.
\end{proof}

\medskip

The manifold $\Emb_{\hor}(M,P)$, factorized by the $\Diff_{\vol}(M)$-action,
determines the manifold (see the appendix) of \textit{horizontal volume submanifolds of $P$ of type $(M,\mu)$},
$$
\Gr_{\hor}^{M,\mu}(P):=\Emb_{\hor}(M,P)/\Diff_{\vol}(M).
$$
The forgetting map defines a projection of  $\Gr_{\hor}^{M,\mu}(P)$ to 
the manifold of horizontal submanifolds of $P$ of type $M$, namely $\Gr_{\hor}^{M}(P):=\Emb_{\hor}(M,P)/\Diff(M)$.

\color{black}
\begin{proposition}[{Case $a=1$}]\label{pp}
Let $(S,\om)$ be a prequantizable symplectic manifold 
and let $(M,\mu)$ be a compact manifold with $H^1(M)=0$ and
total volume equal to 1.
Then the connected components of  $\Gr_{\iso}^{M,\mu}(S)$
are prequantizable coadjoint orbits of $\Diff_{\quant}(P)_0$ $($and also of $ \operatorname{Diff}_{\rm ham}(S)$ when $S$ is compact$)$. The prequantum bundle is $\Gr_{\hor}^{M,\mu}(P)$.
\end{proposition}

\begin{proof}
The candidate prequantum bundle is the circle bundle
\begin{equation}\label{berry}
\bar p: (\Gr_{\hor}^{M,\mu}(P),\al_0)\to( \Gr_{\iso}^{M,\mu}(S),\om_0), \quad\bar p(Q,\nu):=\left(p(Q),p_*\nu\right),
\end{equation}
where the connection 1-form $\al_0$ descends \textcolor{black}{from the 1-form $\bar\al$
on $\Emb_{\hor}(M,P)$ through the quotient map $\pi^\mu$,
while the curvature $\om_0$ is the reduced symplectic form.}
\color{black}

We verify that $\bar\al$ from \eqref{alfa} is a basic form
for the principal  bundle $\Emb_{\hor}(M,S)\to\Gr^{M,\mu}_{\hor}(S)$
with structure group $\Diff_{\vol}(M)$, \ie it is $\Diff_{\vol}(S)$-invariant
and vanishes on infinitesimal generators. Let $F$ be a horizontal embedding, so $F^*\al=0$. Then
\begin{align*}
\bar\al(TF\o u)&=\int_M\al(TF\o u)\mu=\int_M(F^*\al)(u)\mu=0
,\quad u\in\X_{\vol}(M).
\end{align*}
The action of $\ph\in\Diff_{\vol}(M)$ reads $\bar\ph(F)=F\o\ph$, and we have
\begin{align*}
(\bar\ph^*\bar\al)(v_F)
=\bar\al(v_F\o\ph)=\int_M(\al(v_F)\o\ph)\mu
=\int_M\al(v_F)\mu=\bar\al(v_F).
\end{align*}
Hence $\bar\al$ descends to a 1-form $\al_0$ on  $\Gr_{\hor}^{M,\mu}(P)$.

The next step is to show that $\al_0$ is a principal  connection 1-form with curvature 2-form $\om_0$. 
Since the total volume of $M$ is equal to 1, the form $\al_0$ reproduces the infinitesimal generators.
Indeed, for the infinitesimal generator $E_{\Gr}$ of the $S^1$-action on $\Gr_{\hor}^{M,\mu}(P)$, induced by the infinitesimal generator $E\in\X(P)$ of the circle action on $P$, we have
$$\al_0(E_{\Gr})(\pi^\mu(F))=\al_0(T_F\pi^\mu(E\o F))=\bar\al(E\o F)=1.$$
The $S^1$-invariance of $\al_0$ follows from the commutative diagram on the right
\[
\xymatrix{
\Emb_{\hor}(M,P)\ar[d]_{\pi^\mu}\ar[r]^{p_*} &\Emb_{\iso}(M,S) \ar[d]_{\pi^\mu} \\
\Gr_{\hor}^{M,\mu}(P)\ar[r]^{\bar p}& \Gr_{\iso}^{M,\mu}(S)\\
}\qquad  \qquad \xymatrix{
\Emb_{\hor}(M,P)\ar[d]_{\pi^\mu}\ar[r]^{{(\rh_z)}_*} &\Emb_{\hor}(M,P) \ar[d]_{\pi^\mu} \\
\Gr_{\hor}^{M,\mu}(P)\ar[r]^{\bar\rh_z}& \Gr_{\hor}^{M,\mu}(P)\\
}
\]
because $(\pi^\mu)^*(\bar\rh_z^*\al_0)=((\rh_z)_*)^*(\pi^\mu)^*\al_0=((\rh_z)_*)^*\bar\al
=\bar\al=(\pi^\mu)^*\al_0$.
In order to compute the curvature, we use 
the commutative diagram on  the left. 
Recall that $\om_0$ is the reduced symplectic form, so $(\pi^\mu)^*\om_0=\bar\om$.
We get that $(\pi^\mu)^*(\mathbf{d} \al_0)=\mathbf{d} \bar\al=
(p_*)^*\bar\om=(p_*)^*(\pi^\mu)^*\om_0=(\pi^\mu)^*(\bar p^*\om_0)$,
hence the identity $\mathbf{d} \al_0=\bar p^*\om_0$.

This finishes the proof that the circle bundle $\bar p: (\Gr_{\hor}^{M,\mu}(P),\al_0)\to( \Gr_{\iso}^{M,\mu}(S),\om_0)$ is a prequantum bundle.
\end{proof}

\medskip

To treat the general case $\int_M\mu=a\; {\in \mathbb{Z}}$, we will make use of the following simple lemma.


\begin{lemma}
Let $(S,\om)$ be a symplectic manifold,
base of a principal circle bundle $p:P\to S$ endowed with 
an invariant form $\al\in\Om^1(P)$ such that $\dd\al=p^*\om$
and $\al(E)=a\in\ZZ$.
Then $S$ is prequantizable.

A prequantum bundle is 
$p_a:P/{\ZZ_a}\to S$, for the pullback $\ZZ_a$-action on $P$ 
by the natural group homomorphism 
\begin{equation}\label{zea}
\ZZ_a=\ZZ/a\ZZ\to S^1,\quad [j]\mapsto e^{\frac{j}{a}2\pi i},
\end{equation}
endowed  with the principal circle action 
\begin{equation}\label{zact}
z\cdot[y]=[w\cdot y],\text{ where $w^a=z$}.
\end{equation}
\end{lemma}

\begin{proof}
First we check that the action \eqref{zact} is well defined, \ie it doesn't depend on the solution
$w$ of $w^a=z$. Any other solution is of the form $ e^{\frac{j}{a}2\pi i}w$,
so we get the same result 
$$[ e^{\frac{j}{a}2\pi i}w\cdot y]\stackrel{\eqref{zea}}{=}[[j]\cdot(w\cdot y)]=[w\cdot y].$$
We now show that the form $\al_a$ that descends from $\al$ through
the projection $\pi:P\to P/{\ZZ_a}$ is a principal connection. From $\pi^* \mathbf{d} \al_a= \mathbf{d} \al=p^*\om=\pi^*p_a^*\om$, we deduce that $ \mathbf{d} \al_a=p_a^*\om$.
After noticing that the infinitesimal generators $\frac1aE\in\X(P)$ and  $E_a\in\X(P/{\ZZ_a})$ for the action \eqref{zact}
are $\pi$-related, we get that $\al_a(E_a)=\frac1a\al(E)=1$.
Now it easily follows that $\al_a$ is a principal connection.
\end{proof}

\color{black}

\color{black}
\begin{proposition}[{Case $a \in \mathbb{Z}  $}]\label{integra}
Let $(S,\om)$ be a prequantizable symplectic manifold 
with prequantum bundle $P$
and let $(M,\mu)$ be a compact manifold with $H^1(M)=0$ and 
total volume $a\in\ZZ$.
Then the connected components $\Gr_{\iso}^{M,\mu}(S)$
are prequantizable coadjoint orbits of $\Diff_{\quant}(P)_0$ $($and also of $ \operatorname{Diff}_{\rm ham}(S)$ when $S$ is compact$)$. 
The prequantum bundle is $\Gr_{\hor}^{M,\mu}(P)/{\ZZ_a}$.
\end{proposition}
\color{black}


\paragraph{Prequantization via cotangent reduction.}{In this paragraph we consider the special case of Lagrangian submanifolds and verify that the  prequantum bundle \eqref{berry} recovers the Berry bundle constructed in \cite{W90} by cotangent bundle reduction.} 
We are still assuming $H^1(M)=0$, so $T_L \Gr^M_{\Lag}(S)=\dd C^\oo(L)$. Using Lemma \ref{lem4} we can easily describe the tangent space to the space Planckian submanifolds of type $M$, see Lemma 4.1 in \cite{Weinstein}, denoted here by $\Gr^M_{\hor}(P)$. 

\begin{lemma}\label{lemlem}
The tangent space at $Q$ to the manifold of Planckian submanifolds of type $M$ is $T_Q\Gr^M_{\hor}(P)=C^\oo(L)$ for the Lagrangian manifold $L=p(Q)$.
\end{lemma}
\color{black}
\begin{proof}
The linear surjective map 
\begin{equation}\label{isom}
T_Q\Gr^M_{\hor}(P)\to C^\oo(L),\quad[\xi_h|_Q]\mapsto h|_L
\end{equation}
is well defined, where $\xi_h=X_h^{\hor}-(h\o p)E$. 
Indeed, the condition $[\xi_h|_Q]=0$ for some $h\in C^\oo(S)$
implies that $\xi_h|_Q$ is tangent to the Planckian manifold $Q$. Since the connection form $\al$ vanishes on $TQ$, 
we have that $0=\al(\xi_h|_Q)=(h\o p)|_Q$ and consequently $h|_L=0$.
The map \eqref{isom} is also injective: whenever $h|_L=0$, the Hamiltonian vector field $X_h|_L\in \Ga(TL^\om)=\Ga(TL)$ is tangent to the Lagrangian manifold $L$, so $\xi_h|_Q=X_h^{\hor}|_Q$  is tangent to the Planckian manifold $Q$.  We thus get the desired isomorphism \eqref{isom}.
\end{proof}
\color{black}
\medskip

The circle action on the cotangent bundle $T^*\Gr^M_{\hor}(P)$ 
with canonical symplectic form $\Om=\dd\Th$ is Hamiltonian with momentum map
\[
J:T^*\Gr^M_{\hor}(P)\to\RR,\quad J(Q,\nu_{L})=\int_{L}\nu_{L}, \quad L=p(Q),
\]
with $\nu_L\in \Den(L)=T^*_Q\Gr^M_{\hor}(P)$ (by Lemma \ref{lemlem}). 
Let $\mu$ be a volume form on $M$ with total volume $\int_M\mu=1$.
The preimage $J^{-1}(1)\subset T^*\Gr^M_{\hor}(P)$ can be identified with
the nonlinear Grassmannian of weighted Planckian submanifolds
\begin{equation}\label{j11}
J^{-1}(1)=\left\{(Q,\nu_L):\int_L\nu_L=\int_M\mu\right\}
\cong\Gr_{\hor}^{M,\mu}(P),
\end{equation}
since volume forms on $Q$ are in bijection with volume forms on $L=p(Q)$.

The symplectically reduced space at $1$ is the nonlinear Grassmannian of weighted Lagrangian submanifolds
\[
J^{-1}(1)/{S^1}=\Gr_{\hor}^{M,\mu}(P)/{S^1}=\Gr_{\Lag}^{M,\mu}(S),
\]
endowed with reduced symplectic form $\Om_1$ that
descends from the restriction to $J^{-1}(1)$ of
the canonical cotangent bundle symplectic form $\Om=\dd\Th$.
Thus we also have a prequantum bundle
\begin{equation}\label{ctg}
(J^{-1}(1),\Th)\to (\Gr_{\Lag}^{M,\mu}(S),\Om_1).
\end{equation}

\begin{proposition}
The prequantum bundles \eqref{berry} and \eqref{ctg} coincide
\textcolor{black}{for $\dim M=\frac12\dim S$.}
\end{proposition}
\begin{proof}
It suffices to show that the 1-form $\al_0$ on $\Gr_{\hor}^{M,\mu}(P)$, 
naturally induced by the connection $\al\in\Om^1(P)$, 
as in Proposition \ref{pp}, corresponds under \eqref{j11} to the restriction of the canonical 1-form 
$\Th$ on the cotangent bundle $T^*\Gr^M_{\hor}(P)$ 
to the level set $J^{-1}(1)$.

In the computation we make use of the transitivity of the $\X_{\quant}(P)$-action 
on $\Emb_{\hor}(M,P)$, hence on $\Gr_{\hor}^{M}(P)$, but also on $\Gr_{\hor}^{M,\mu}(P)\cong J^{-1}(1)$.
By Lemma \ref{lemlem}, the infinitesimal generator of $\xi_h$ at $Q\in\Gr_{\hor}^{M}(P)$
can be identified with the restriction of $h\in C^\oo(S)$ to $L=p(Q)$. For all
$F\in \Emb_{\hor}(M,P)$ with $f=p\o F$, $F(M)=Q$, $f_*\mu=\nu_L$, we get
\begin{align*}
(\al_0)_{(Q,\nu_L)} \big( (\xi_h)_{\Gr_{\hor}^{M,\mu}} \big) &=
\bar\al_F(\xi_h\o F)=\int_M(\al(\xi_h)\o F)\mu=\int_M(h\o f)\mu
=\int_Lh|_L\nu_L\\
&=(\nu_L,(\xi_h)_{\Gr_{\hor}^{M}}(Q))=\Th_{(Q,\nu_L)}\big((\xi_h)_{T^*\Gr_{\hor}^{M}}\big),
\end{align*}
where we denote by $(\xi_h)_{\Gr_{\hor}^{M,\mu}}$ and $(\xi_h)_{T^*\Gr_{\hor}^{M}}$ the infinitesimal generators for the 
$\X_{\quant}(P)$-action on $\Gr_{\hor}^{M,\mu}(P)$ and on
$T^*\Gr_{\hor}^{M}(P)$.
\end{proof}

\medskip

As a consequence
the explicit formula for the connection 1-form $\Th$ on $J^{-1}(1)$ in terms of the tangent space decomposition
$T_{(Q,\nu_L)}J^{-1}(1)=C^\oo(L)\x \dd\Om^{m-1}(L)$ is simply
\[
\Th_{(Q,\nu_L)}(h,\dd\la)=\int_Lh\nu_L,
\]
while the formula for the curvature 2-form $\Om_1$ on $\Gr_{\Lag}^{M,\mu}(S)$
 in terms of the tangent space decomposition
$T_{(L,\nu_L)}\Gr_{\Lag}^{M,\mu}(S)=\dd C^\oo(L)\x \dd\Om^{m-1}(L)$
is the same as  \eqref{patru}:
\[
(\Om_1)_{(L,\nu_L)}((\dd h_1,d\la_1),(\dd h_2,\dd\la_2))=\int_L(\dd h_1\wedge\la_2-\dd h_2\wedge\la_1).
\]



\section{Exact isotropic submanifolds}\label{s8}

\paragraph{Dual pair.} 
In the special case when the symplectic form $ \omega $ on $S$ is exact, 
i.e. $ \omega = - \mathbf{d} \theta $, 
an ideal fluid dual pair exists on $\Emb(M,S)$ without the extra condition $H^1(M)=0$. 
The left acting group is still $ H=\Diff_{\quant}(P)_0$ with momentum map $\J_L(f)=f_*\mu$.  
In this case the prequantization central extension 
$\Diff_{\quant}(P)_0$ is defined 
with the group cocycle $B( \eta _1 , \eta _2 )= \int_{x_0 }^{ \eta _2 (x_0 )}( \eta ^\ast _1 \theta  - \theta  )$, see \cite{IsLoMi2006}. The cohomology class of this cocycle is independent of the chosen point $x_0 \in S$. 
For the right action on $\Emb(M,S)$  it is not needed to restrict to the subgroup $ \operatorname{Diff}_{\rm ex}(M)$ and to consider the Ismagilov central extension.
One  takes $G=\operatorname{Diff}_{\rm vol}(M)$
and the momentum map is $ \mathbf{J} _R^{\rm ex} (f)=[f ^\ast\theta ]$.
The pair of momentum maps reads
\begin{equation}\label{MW_GBVi_dual_pair_exact} 
\Den_c(S)=C^\infty(S)^\ast \stackrel{\J_L}{\longleftarrow} \operatorname{Emb}(M,S)\stackrel{\J_R^{\rm ex}}{\longrightarrow} \mathfrak{X}_{\rm vol}(M) ^\ast = \Omega ^1 (M)/ \mathbf{d} \Omega ^0 (M).
\end{equation}

\begin{lemma}[\cite{GBVi2014}]\label{lem3} The group $\Diff_{\ham}(S)$ acts transitively on each of the connected components of the level sets of the momentum map $\J^{\ex}_R$
for the $\Diff_{\vol}(M)$-action on $\Emb(M,S)$.
\end{lemma}
Since the left action of $H=\Diff_{\quant}(P)_0 $ is given by composition on the left by $ \operatorname{Diff}_{\rm ham}(S)$, from Lemma \ref{lem1} and Lemma \ref{lem3} we obtain that the actions of $\Diff_{\rm vol}(M)$ and $ \Diff_{\quant}(P)_0 $ on $ \operatorname{Emb}(M,S)$ are mutually completely orthogonal. We thus obtain the following result.

\begin{proposition}[\cite{GBVi2014}]\label{dual_pair_exact} The pair of momentum maps \eqref{MW_GBVi_dual_pair_exact} is a dual pair associated to mutually completely orthogonal actions. 
\end{proposition}

\paragraph{Exact isotropic embeddings.}
The zero level set of the momentum map $\J_R^{\ex}$
is the manifold of \textit{exact isotropic embeddings}
\[
(\J_R^{\rm ex})^{-1}(0)=\{f\in\Emb(M,S): [f^*\theta ]=0\}=:\Emb_{\iso,\ex}(M,S)
\]
which consists of embeddings $f$ such that $f^*\theta$ is an exact 1-form. 
The image $f(M)$ is an {\it exact isotropic submanifold} of $S$ of type $M$.
The nonlinear Grassmannian $\Gr_{\iso,\ex}^M(S)$ of exact isotropic submanifolds of type $M$
is a submanifold of $\Gr_{\iso}^M(S)$ and 
the base of a principal $\Diff(M)$-bundle $\Emb_{\iso,\ex}(M,S)\to\Gr_{\iso,\ex}^M(S)$.
Using the dual pair property of \eqref{MW_GBVi_dual_pair_exact}, one verifies that the tangent space to $\Gr_{\iso,\ex}^M(S)$ has the expression \eqref{ting}. The notions of isotropic and exact isotropic embeddings coincide
if $H^1(M)=0$.

The reduced symplectic manifold at zero for the $\Diff_{\vol}(M)$-action
can be identified with the manifold 
of exact isotropic volume submanifolds of type $(M,\mu)$:
$$(\J_R^{\ex})^{-1}(0)/\Diff_{\vol}(M)=\Emb_{\iso,\ex}(M,S)/\Diff_{\vol}(M)=\Gr_{\iso,\ex}^{M,\mu}(S).$$

\paragraph{Coadjoint orbits.} We use again
the link between symplectic reduction in dual pairs of momentum maps
and coadjoint orbits to obtain coadjoint orbits of the quantomorphism group
from  the dual pair  \eqref{MW_GBVi_dual_pair_exact}.

\color{black}
\begin{theorem}\label{tre}
Let $(S,-\dd\th)$ be an exact symplectic manifold.
Each connected component $\Gr_{\iso,\ex}^{M,\mu}(S)_0$ of  $\Gr_{\iso,\ex}^{M,\mu}(S)$
is symplectically diffeomorphic to a coadjoint orbit of $\Diff_{\quant}(P)_0 $ 
with $P=S\x S^1$ and $\al=\dd t-p^*\th$.
The diffeomorphism is given by
\[
\bar \J_L: \Gr_{\iso,\ex}^{M,\mu}(S)_0 \rightarrow \mathcal{O} \subset \operatorname{Den}_c(S), \quad
{ \langle \bar\J_L(N, \nu ),h\rangle=\int_N h|_N\nu.}
\]
The orbit symplectic form on $\Gr_{\iso,\ex}^{M,\mu}(S)$ has the same expression \eqref{trei} as in Theorem \ref{formula_KKS}.
\end{theorem}
\color{black}
\begin{proof}  The result follows from Proposition \ref{id1} applied to the ideal fluid dual pair \eqref{MW_GBVi_dual_pair_exact}. Indeed, from Theorem \ref{dual_pair_exact}, we know that the actions are completely mutually orthogonal. From Lemma \ref{lem1}, the action of $G= \operatorname{Diff}_{\rm vol}(M)$ is transitive on the level sets of $ \mathbf{J} _L$. Finally, from Lemma \ref{lem3}, the action of $H=\Diff_{\quant}(P)_0$ (connected) is transitive on each connected component of the level sets of $ \mathbf{J} _R^{\rm ex}$. 
\end{proof}

\medskip

{This result is relevant only in the case $H ^1 (M)\neq 0$, since if $H ^1 (M)=0$ it is a particular instance of previous results.}
 
\paragraph{Exact Lagrangian submanifolds.}
Let $M$ be compact with $\dim M=\frac12\dim S$ and no condition on $H^1(M)$.
Let $L$ be an exact Lagrangian submanifold of $(S,-\dd\th)$ of type $M$.
Under the identification $T_L\Gr^M(S)=\Om^1(L)$,
the tangent spaces to the nonlinear Grassmannian of
(exact) Lagrangian submanifolds become
$T_L\Gr^M_{\Lag}(S)=Z^1(L)$ and $T_L\Gr^M_{\Lag,\ex}(S)=\dd C^\oo(L)$.

As in example 5.32 in \cite{MDSa98} (see also  \cite{Donaldson} and \cite{Hitchin})
we obtain the space $\Emb_{\Lag,\ex}(M,S)$ 
of exact Lagrangian embeddings as the zero set
of the momentum map $ \mathbf{J} _R ^{\rm ex}$.
The reduced symplectic manifold can be identified with the manifold $\Gr_{\Lag,\ex}^{M,\mu}(S)$
of exact Lagrangian volume submanifolds of type $(M,\mu)$ in $S$.
Theorem \ref{tre} ensures that each connected component of this manifold
is a coadjoint orbit of the quantomorphism group.

The basic example is the cotangent bundle $S=T^*M$ with canonical symplectic form, where the image of an exact 1-form $\mathbf{d}h$, $h\in C^\oo(M)$, viewed as a section of $T^*M$,
is an exact Lagrangian submanifold of $T^*M$.

\paragraph{Prequantization.}
Recall that the prequantum bundle over $(S,-\dd\th)$ is $P=S \times  S^1  $ with connection 1-form  $ \alpha = dt- p ^\ast \theta $. Given $F \in \operatorname{Emb}(M,P)$, we shall use the notation $F=(f,g)$, $f\in C^\infty(M,S)$, $g \in C^\infty(M,  S^1  )$.
We define the manifold of \textit{exact horizontal embeddings}
\[
\operatorname{Emb}_{\rm hor, ex}(M,P)=\{(f,g) : f ^\ast \theta = g ^{-1} \mathbf{d} g\;\; \in \mathbf{d} C^\infty(M) \}.
\]
Note that the condition $ f ^\ast \theta = g ^{-1} \mathbf{d} g$  is equivalent to $ F ^\ast \alpha =0$.

\begin{lemma}
The circle action on $ \Emb_{\hor, \ex}(M,P)$
induced by the principal circle action on $P$ leads to the principal circle bundle
\[
p_*: \Emb_{\hor, \ex}(M,P)\to\Emb_{\iso, \ex}(M,S), \quad p_*(F)=p\o F.
\]
\end{lemma}
\begin{proof} The proof goes as for Lemma \ref{lemma_emb_hor}, except for the surjectivity of $p _\ast $ that we now show.
Let $f\in \operatorname{Emb}_{\iso, \ex}(M,S)$ so $ f ^\ast \theta = \mathbf{d} h$, where $h \in C^\infty(M)$. We define $g:=e^{ {\rm i} h} \in C^\infty(M,  S^1  )$. Because $g^{-1} \mathbf{d} g= \mathbf{d} h$, the embedding $F:=(f,g)\in \operatorname{Emb}_{\rm hor, \rm ex}(M,P)$ descends to $f$.
\end{proof}

\medskip

$\Emb_{\hor, \ex}(M,P)$ factorized by the $\Diff_{\vol}(M)$-action
determines the manifold of \textit{exact horizontal volume submanifolds of $P$ of type $(M,\mu)$},
$$
\Gr_{\hor, \ex}^{M,\mu}(P):=\Emb_{\hor, \ex}(M,P)/\Diff_{\vol}(M).
$$
We have a similar result to Proposition \ref{integra}.

\color{black}
\begin{proposition}\label{prequant_exact} 
Let $(S,- \mathbf{d} \theta )$ be an exact symplectic manifold, 
and let $(M,\mu)$ a compact manifold with total volume $a\in\ZZ$.
Then the connected components of  $\Gr_{\iso, \ex}^{M,\mu}(S)$
are prequantizable coadjoint orbits of $\Diff_{\quant}(P)_0$.
The prequantum bundle is $\Gr_{\hor,\ex}^{M,\mu}(P)/{\ZZ_a}$.
\end{proposition}
\color{black}


\paragraph{Circle embeddings.} All embeddings of the circle $M=S^1$ 
\textcolor{black}{(endowed with the standard volume form)} into $(S,-\dd\th)$ are isotropic,
but not all of them are exact isotropic. The right momentum map 
for the action of $\Diff_{\vol}(S^1)=\Rot(S^1)$ becomes
\[
\J^{\ex}_R:\Emb(S^1,S)\to \Om^1(S^1)/\dd\Om^0(S^1)\cong\RR, \quad\J^{\ex}_R(f)=\int_{S^1}f^*\th,
\]
so that symplectic reduction at zero gives the quotient space
\[
\left\{f\in\Emb(S^1,S):\int_{S^1}f^*\th=0 \right\}/\Rot(S^1).
\]
If $S$ is simply connected, then we get the space of circle embeddings
that enclose a piece of surface in $(S,-\dd\th)$ of vanishing symplectic area,
modulo circle rotations.
Its connected components are coadjoint orbits of the quantomorphism group,
by Theorem \ref{tre}.
\textcolor{black}{Moreover, they are also prequantizable since the total volume of  the circle is 1.}

\color{black} 

\section{Conclusion}

In this paper we have presented a systematic way to identify and describe a class of coadjoint orbits of the group of Hamiltonian diffeomorphisms of a symplectic manifold $(S,\om)$. Our systematic approach takes advantage of the properties of a dual pair of momentum maps arising in fluid dynamics, presented in \cite{Marsden-Weinstein} and studied in \cite{GBVi2011}. These properties are used for the identification of symplectic reduced space with coadjoint orbits. The symplectic reduced spaces turn out to be nonlinear Grassmannians (manifolds of submanifolds) with additional geometric structures.
By implementing symplectic reduction at the zero momentum we obtained the identification of coadjoint orbits with Grassmannians of isotropic volume submanifolds, {slightly} generalizing the results in \cite{W90} and \cite{L}. At the other extreme, by implementing symplectic reduction at a nondegenerate momentum we obtained the identification of connected components of the nonlinear symplectic Grassmannian with coadjoint orbits, thereby recovering the result of \cite{Haller-Vizman}. \textcolor{black}{We also commented on the intermediate cases which correspond to new classes of coadjoint orbits.}

The dual pair property also turned to be advantageous to concretely describe the tangent spaces to these Grassmannians a well as the orbit symplectic forms.
We have also shown that, whenever the symplectic manifold $(S,\om)$ is prequantizable, the coadjoint orbits that consist of isotropic submanifolds with  total volume $a\in\ZZ$ are prequantizable, by constructing explicitly the prequantum bundle. This result extends previous results obtained in \cite{W90} for  Lagrangian submanifolds. Finally, we considered in details the case of an exact symplectic manifold, and showed that it allows the treatment of submanifolds with nontrivial cohomology.

\color{black}
 

\appendix

\section{Smooth structures on nonlinear Grassmannians}\label{ax1}

In this appendix we show that several sets of submanifolds can be endowed with natural 
smooth structures modeled on Fr\'echet spaces. This is done
by an explicit construction of manifold charts.
{The manifolds $M$ and $S$ are assumed to be finite dimensional.}

\paragraph{The nonlinear Grassmannian of submanifolds of type $M$.} 
If $M$ is a submanifold of the Riemannian manifold $(S,g)$,
then the normal bundle $TS|_M/TM$
can be identified with the orthogonal subbundle $TM^{\per}\subset TS|_M$.  
A normal tubular neighborhood $U\subset S$ of $M$ can be built with the exponential map, namely with the diffeomorphism 
\begin{equation}\label{has}
h=\exp:V\subset TM^\per\to U\subset S
\end{equation}
that coincides with the identity on $M$.

An embedding of $M$ into $S$, obtained from a normal vector field $s\in\Ga(TM^\per)$ via the tubular neighborhood diffeomorphism as
$f^\perp=h\o s$, will be called a {\it normal embedding}.
Let $\U$ denote the set of all submanifolds of $S$ that are images of normal embeddings.
We can recover the normal embedding from its image submanifold $N\in\U$ 
with the help of an arbitrary embedding $f:M\to S$ such that $f(M)=N$,
namely $f^\perp=f\o\ps_f^{-1}$. Here $\ps_f=p\o h^{-1}\o f$
is a diffeomorphism of $M$, with $p: TS \rightarrow S$ denoting the projection.
It doesn't depend on the choice of the embedding $f$: 
starting with another embedding of $M$ onto $N\subset S$, $f\o\ps$
with $\ps\in\Diff(M)$,
we get the same normal embedding $f^\perp$ because $\ps_{f\o\ps}=\ps_f\o\ps$.
Of course the normal embedding for $N$ depends on the way $M$ sits in $S$.

We recall from  \cite{KrMi97} the construction of charts for $\Gr^M(S) = \Emb(M,S)/\Diff(M)$  using normal tubular neighborhoods. A chart around $M$ is defined on $\U\subset\Gr^M(S)$ with values in the Fr\'echet neighborhood of the zero section in $ \Gamma (TM ^\per)$
consisting of $V$-valued sections by
\begin{equation}\label{chi0}
\chi(N)=h^{-1}\o f^\perp
\;\text{ with inverse }\;\chi^{-1}(s)=(h\o s)(M),
\end{equation}
for $f^\perp$ the unique normal embedding such that $f^\perp(M)=N$.

Bundle charts on the $\Diff(M)$-bundle $\pi:\Emb(M,S)\to\Gr^M(S)$, $\pi(f)=f(M)$, 
can be taken of the following form:
\begin{equation}\label{square}
\tilde{ \chi }:  \pi^{-1}(\mathcal{U}) \subset \operatorname{Emb}(M,S) \rightarrow \mathcal{U}\x\Diff(M)\subset  \Gr^M(S) \times \Diff (M) , \quad 
\tilde{ \chi }(f)=( \pi(f) ,  \psi_f),
\end{equation}
where $ \psi_f= p \circ h^{-1} \circ f$. Its inverse reads
$\tilde{ \chi } ^{-1} ( N , \psi)= f^\perp \circ\psi$,
for $f^\perp$ the unique normal embedding such that $f^\perp(M)=N$.

\paragraph{The nonlinear Grassmannian of volume submanifolds.}
We build bundle charts on 
$\Gr^{M, \vol}(S)=\left\{(N,\nu ):N\in \Gr^M(S),\nu\in\Vol(N)\right\}$,
the nonlinear Grassmannian of volume submanifolds of type $M$,
with projection the forgetting map 
$$\pi_1:(N, \nu )\in \Gr^{M,\vol}(S) \mapsto N\in \Gr^M(S),$$ 
is a fiber bundle with fiber $ \operatorname{Vol}(M)$.
Bundle charts can be written using again the unique normal embedding
$f^\perp$ such that $f^\perp(M)=N$:
\begin{equation}\label{circle}
\bar{ \chi }: \pi_1^{-1}( \mathcal{U}) \subset \Gr^{M,\vol}(S) \rightarrow \mathcal{U}\x\Vol(M)\subset  \Gr^M(S) \times \Vol(M), \quad
 \bar{ \chi }(N,\nu)=(N,(f^\perp)^*\nu).
\end{equation}
Its inverse reads
$\bar{ \chi } ^{-1} ( N , \mu)= (N,(f^\perp)_*\mu)$.
In this way we get also manifold charts on the nonlinear Grassmannian 
of volume submanifolds.

The natural map 
\begin{equation}\label{voll}
\vol:\Gr^{M,\vol}(S)\to \RR,\quad\vol(N,\nu)=\int_N\nu
\end{equation}
is a submersion.
This can be seen in the coordinate chart of type \eqref{circle} since $\vol\o\bar\chi=\int_M\o\pr_2:\U\x\Vol(M)\to\RR$.

A regular value theorem (Theorem III.11 in \cite{NW}) extracted from Gloeckner's
implicit function theorem  (Theorem 2.3 in \cite{Gloeckner}) states that,
given a smooth map $F:\M\to \N$ into a Banach manifold $\N$, if
there is a continuous linear splitting of each tangent map $T_xF$ 
for $F(x)=y_0$, then $F^{-1}(y_0)$ is a submanifold of $\M$.
It can be applied to the submersion \eqref{voll} to show that the nonlinear Grassmannian 
$\Gr^{M,\mu}(S)$ of volume submanifolds of type $(M,\mu)$,
the preimage of the total volume $\int_M\mu$ by the map $\vol$,
is a submanifold of $\Gr^{M,\vol}(S)$.

\paragraph{The nonlinear Grassmannian of isotropic submanifolds.}
A manifold structure can be defined on the nonlinear Grassmannian $\Gr_{\iso}^M(S)$ 
of type $M$ isotropic submanifolds of $S$, 
by restriction of the manifold charts $ \chi$ of $ \Gr^M(S)$.
First we rescale the tubular neighborhood diffeomorphism \eqref{has}
to a diffeomorphism $h:TM^\per\to U\subset S$,
then we define $\om_{M}:={h}^*\om$ and the linear subspace
$$
\Ga_{\iso}(TM^\per)=\{s\in\Ga(TM^\per):s^*\om_{M}=0\}.
$$ 
The chart $\chi$ is a submanifold  chart:
$N\in\Gr_{\iso}^M(S)$ if and only if
$s=\chi(N)=h^{-1}\o f^\perp  \in\Ga_{\iso}(TM^\per)$,
since $s^*\om_{M}=(f^\perp)^*\om$.

The charts $\tilde\chi$ 
from \eqref{square} for the principal bundle of embeddings restrict to principal bundle charts on
the principal $\Diff(M)$-bundle of isotropic embeddings 
$\Emb_{\iso}(M,S)\to\Gr_{\iso}^M(S)$.
Endowed with the induced smooth structure, $\Emb_{\iso}(M,S)$ becomes
a submanifold of $\Emb(M,S)$.

\paragraph{The nonlinear Grassmannian of isotropic volume submanifolds.}
The manifold structure on the space
of isotropic volume submanifolds $\Gr_{\iso}^{M,\mu}(S)=\Emb_{\iso}(M,S)/\Diff_{\vol}(M)$ is presented in \cite{L}.
We proceed similarly to $\Gr^{M,\mu}(S)$.
We use the bundle charts $\bar\chi$ of $\Gr^{M,\vol}(S)$ from \eqref{circle}
to show that the space of all isotropic volume submanifolds of type $M$ 
(with arbitrary volume forms on $M$ allowed), denoted by
$\operatorname{Gr}_{\iso}^{M, \vol }(S)$, is a submanifold of $\Gr^{M,\vol}(S)$.
Then we use the restriction of the submersion $\vol$ to the submanifold $\Gr_{\iso}^{M,\vol}(S)$
to show that $\Gr_{\iso}^{M,\mu}(S)$ is a submanifold.
A similar reasoning works for the exact versions:  $\operatorname{Gr}_{\iso,\ex}^M(S)$,
$\operatorname{Gr}_{\iso, \ex}^{M, \vol}(S)$, and $\operatorname{Gr}_{\iso, \ex}^{M, \mu}(S)$, using the linear subspace 
$ \Gamma _{\iso,\ex}(TM^\per)$ instead of $ \Gamma _{\iso}(TM^\per)$,
and also for the exact version  $\operatorname{Gr}_{\hor,\ex}^{M, \mu }(P)$.


\paragraph{The nonlinear Grassmannian of horizontal submanifolds.}
Let  $p:(P,\al)\to (S,\om)$ be a prequantum bundle
with connection $\al$ and curvature the symplectic form $\om$.
Using the rescaled tubular neighborhood diffeomorphism $h:TM^\per\to U\subset P$ 
for the submanifold $M$ of $P$, 
we define the 1-form $\al_{M}:={h}^*\al$ on $TM^\per$ and the linear subspace
$$
\Ga_{\hor}(TM^\per)=\{s\in\Ga(TM^\per):s^*\al_{M}=0\}.
$$ 
Now the chart $\chi$ of $\Gr^M(P)$ is a submanifold  chart for $\Gr_{\hor}^M(P)$:
$Q\in\Gr_{\hor}^M(P)$ if and only if
$s=\chi(Q)=h^{-1}\o F^\perp  \in\Ga_{\hor}(TM^\per)$,
since $s^*\al_{M}=(F^\perp)^*\al$.

The principal bundle
charts $\tilde\chi$ for $\Emb(M,P)$ from \eqref{square} 
restrict to principal bundle charts on
the principal $\Diff(M)$-bundle of horizontal embeddings 
$\Emb_{\hor}(M,P)\to\Gr_{\hor}^M(P)$. Moreover, $\Emb_{\hor}(M,P)$
becomes a submanifold of $\Emb(M,S)$.

We use the bundle charts $\bar\chi$ of $\Gr^{M,\vol}(P)$
to show first that the space  $\Gr_{\hor}^{M,\vol}(P)$
of horizontal volume submanifolds of type $M$ 
(with arbitrary volume forms on $M$ allowed) is a submanifold of $\Gr^{M,\vol}(P)$.
Then we use the restriction of the submersion \eqref{voll} 
to it, to show that the nonlinear Grassmannian $\Gr_{\hor}^{M,\mu}(P)$
of  horizontal volume submanifolds of type $(M,\mu)$ is a submanifold
of $\Gr_{\hor}^{M,\vol}(P)$.

\paragraph{Acknowledgments.}
\textcolor{black}{We are grateful to the referee for very useful suggestions that improved the presentation of the paper.
}
Both authors were partially supported by the LEA Franco-Roumain ``MathMode"
and by the Erwin Schr\"odinger Institute in Vienna. 
Fran\c{c}ois Gay-Balmaz was also partially supported by the ANR project GEOMFLUID, ANR-14-CE23-0002-01. Cornelia Vizman was
also partially supported by CNCS UEFISCDI, project number PN-II-ID-PCE-2011-3-0921.

{
\footnotesize

\bibliographystyle{new}

}
\end{document}